\documentclass[11pt]{article}
\usepackage[T2A]{fontenc}
\usepackage[utf8]{inputenc}
\usepackage[russian]{babel}
\usepackage{graphicx}
\usepackage{amsfonts,amssymb}
\usepackage{latexsym}
\usepackage{amssymb, amsmath, textcomp, amsthm, multicol}
\newtheorem{theorem}{Теорема}

\newtheorem{statement}{Утверждение}

\newtheorem{comment}{Замечание}
\newtheoremstyle{neosn}{0.5\topsep}{0.5\topsep}{\rm}{}{\sc}{.}{ }{\thmname{#1}\thmnumber{ #2}\thmnote{ {\mdseries#3}}}
\theoremstyle{neosn}
\newtheorem{definition}{Определение}

\graphicspath{{noiseimages/}}
\DeclareGraphicsExtensions{.pdf,.png,.eps}
 \usepackage{graphicx,color}
\usepackage[left=4cm,right=2cm,
top=2cm,bottom=2cm,bindingoffset=0cm]{geometry}
\usepackage{hyphenat}

\begin{document} 
		
		\begin{center}
			\hfill \break
			
			\footnotesize{ФЕДЕРАЛЬНОЕ ГОСУДАРСТВЕННОЕ БЮДЖЕТНОЕ ОБРАЗОВАТЕЛЬНОЕ}\\ 
			\footnotesize{УЧРЕЖДЕНИЕ ВЫСШЕГО ОБРАЗОВАНИЯ}\\
			\small{<<МОСКОВСКИЙ ГОСУДАРСТВЕННЫЙ УНИВЕРСИТЕТ}\\
			\small{имени М.В. ЛОМОНОСОВА>>}\\
			\hfill \break
			\normalsize{МЕХАНИКО-МАТЕМАТИЧЕСКИЙ ФАКУЛЬТЕТ}\\
			\hfill \break
			\normalsize{КАФЕДРА ДИФФЕРЕНЦИАЛЬНОЙ ГЕОМЕТРИИ И ПРИЛОЖЕНИЙ}\\
			\hfill\break
			\hfill \break
			\large{ВЫПУСКНАЯ КВАЛИФИКАЦИОННАЯ РАБОТА} \\
			\large{(ДИПЛОМНАЯ РАБОТА)} \\
			\large{специалиста} \\
			\hfill \break 
			\large{\textbf{ОСОБЕННОСТИ СЛОЕНИЙ ЛИУВИЛЛЯ ДЛЯ БИЛЛИАРДОВ В НЕВЫПУКЛЫХ ОБЛАСТЯХ}}\\
			\hfill \break
			\hfill \break
			\hfill \break
		\begin{flushright}
				Выполнил студент \\
			607 академической группы \\
			Москвин Виктор Александрович \\
		\end{flushright}
			\hfill \break
		\begin{flushright}
			\underline{\hspace{5cm}} \\
			подпись студента \\
		\end{flushright}
			\hfill \break
			\begin{flushright}
			Научные руководители: \\
			Академик Анатолий Тимофеевич Фоменко
		\end{flushright}
		\begin{flushright}
			\underline{\hspace{5cm}} \\
			подпись научного руководителя \\
		\end{flushright}
	\begin{flushright}
		
		Ассистент Ведюшкина Виктория Викторовна \\
		\underline{\hspace{5cm}} \\
		подпись научного руководителя \\
	\end{flushright}
		\hfill \break
		\end{center}

		\hfill \break
		\hfill \break
		\begin{center} Москва \\ 2020 \end{center}
		\thispagestyle{empty} 
		
		
		\newpage
		
		\tableofcontents 
		\newpage
		
		\newpage
	\section{Введение.}
	Математический биллиард --- динамическая система, опи\-сы\-вающая движение без трения материальной точки внутри области с абсолютно упругим отражением от границы (угол падения равен углу отражения). 
	В книге С.Л. Табачникова [1] дан обзор сов\-ремен\-ных исследований биллиардов. В настоящей работе ис\-следуют\-ся плоские биллиарды, ограниченные дугами софокусных квадрик, причем допускается, что области могут быть невыпуклыми, т.е. границы биллиардов могут содержать вершины углов $3\pi/2$. Интегрируемость таких биллиардов была замечена В.В. Козловым и Д.В. Трещевым в [3]: первый интеграл это полная энергия, сохраняющаяся в силу отсутствия трения, а дополнительный интеграл --- это параметр каустики. Каустика --- это софокусная квадрика семейства, обладающая тем свойством, что если траектория касается каустики в одной точке, то она обязана касаться той же каустики после отражения от границы биллиарда. 
	
	Цель данной работы --- исследовать топологию изоэнергетической поверхности биллиарда $Q^3$, неком\-пакт\-ного 3-многообразия пар вида $(x,v)$, где $x$ --- координата точки на плоскости, а $v$ --- единичный вектор скорости точки. Изоэнергетическое многообразие расслоено на совместные поверхности уровни интегралов системы, где почти все уровни являются регулярными, а конечное число уровней является критическими. Теорема Лиувилля [1] для биллиардов, ограниченных дугами софокусных квадрик без невыпулых углов на границе, гарантировала, что регулярные совместные поверхности уровня интегралов --- это двумерные торы, а В.В. Ведюшкиной [4 --- 7] были также исследованы и критические уровни слоения Лиувилля, которые описывались с помощью атомов А.Т. Фоменко [2]. Однако, в силу того, что в вершинах невыпуклых углов невозможно корректно определить биллиардное отражение, сохранив непрерывность системы, потоки в таких биллиардах уже не будут полны, а это приведет к невыполнимости условий теоремы Лиувилля, так как невозможно корректно определить отражение в вершине $3\pi/2$, сохранив непрерывность системы. В. Драгович и М. Раднович [8-12] представили описание всех регулярных совместных поверхностей уровня интегралов для невыпуклых биллиардов: для всех значений интеграла в таких биллиардах связная компонента совместной поверхности уровня будет сферой с ручками и проколами. 
	
	Теперь, для дальнейшего исследования топологии изоэнергетической поверхности $Q^3$ нужно описать все критические интегральные поверхности, а также, описать их окрестности, т.е. описать как именно критические поверхности уровня (не являющиеся сферами с ручками и проколами) приклеиваются к регулярным. Ранее автором было представлено полное описание топологии слоений Лиувилля в невыпуклых биллиардах с единственным углом $3\pi/2$ (см. [18]), а в работе [19] также автором был представлен алгоритм построения двумерных особых слоев. В настоящей работе исследована топология фазового многообразия в окрестностях критичес\-ких значений интеграла, а именно, дано описание трехмерных окрестностей двумерных комплек\-сов, являющихся прообразами критических значений дополнительного интеграла. Дальнейшее исследования результатов теории математических биллиардов и теории интег\-рируемых систем см. в [7---18]. 
	
	Первый тип критических уровней тради\-ционно возникает в биллиардных системах и при переходе через этот уровень меняется тип биллиардной каустики --- эллиптический на гиперболический. Второй тип критических уровней является уникальным для невыпуклых биллиардов и возникает, когда каустика проходит через вершины углов $3\pi/2$. Можно заметить, что при переходе через любой из бифуркационных уровней меняется топология совместных поверхностей уровня интег\-ралов: меняется количество и род поверхностей, лежащих в прообразах регулярных зна\-чений интеграла. В главе 5 описана топология окрес\-тностей бифуркационнных слоев в окрес\-тности не седлового уровня, а в главе 7 --- седлового. 
	
	Автор приносит благодарность А.Т. Фоменко за постановку задачи и внимание к работе, В.В. Ведюш\-киной --- за многочисленные ценные обсуждения.
	\subsection{Определение биллиарда.}
	Рассмотрим динамическую систему, описывающую движение (материальной) точки внут\-ри области $\varOmega$ с естественным отражением на границе $P = \partial \varOmega$. Эту систему назовём бил\-лиардом в области. Траектории, попавшие в прямые углы, мы доопределим, как обычно, по непре\-рывности (попадая в вершину прямого угла, точка отражается по той же траектории). Легко видеть, что поступить так же с траекториями, попавшими в вершину угла $3\pi/2$, сохраняя при этом непрерывность системы (свойство, что близкие траектории после отражения остаются близкими), невозможно: так как предел траекторий справа и предел траекторий слева не совпадает (см. рис. \ref{BoundRef}). Далее будем считать, что если траектория попала в вершину угла $3\pi/2$, то она заканчивается в этой вершине. Обозначим вершины углов $\pi/2$ через $x_t$, а углов $3\pi/2$ через --- $x_k$. Следовательно, некомпактным фазовым пространством данного биллиарда является многообразие 
	\begin{center}
		$M^4 := \{{(x, v)| x \in \varOmega, \ x \neq x_k \ \forall k, \ v \in  T^2_x\mathbb{R}, \ |v| > 0}/ \sim\}$,
	\end{center}
	где отношение эквивалентности задается так: \\
	$$(x_1 , v_1 ) \sim (x_2 , v_2 ) \Leftrightarrow x_1 = x_2 \in P, \
	|v_1 | = |v_2 |, \ v_1 - v_2 \bot T_{x_1} P$$
	$$(x_t, v_i) \sim (x_t, v) \Leftrightarrow |v| = |v_i|, v + v_i = 0 \ldotp $$ 
	Здесь через $T_x P$ обозначена касательная прямая к границе области $\varOmega$ в точке $x$, а через $|v|$ евклидова длина вектора $v$. Это отношение эквивалентности иногда будем называть биллиардным
	законом.
	 \begin{figure}[!ht]
		\centering
		\includegraphics[width=0.7\textwidth]{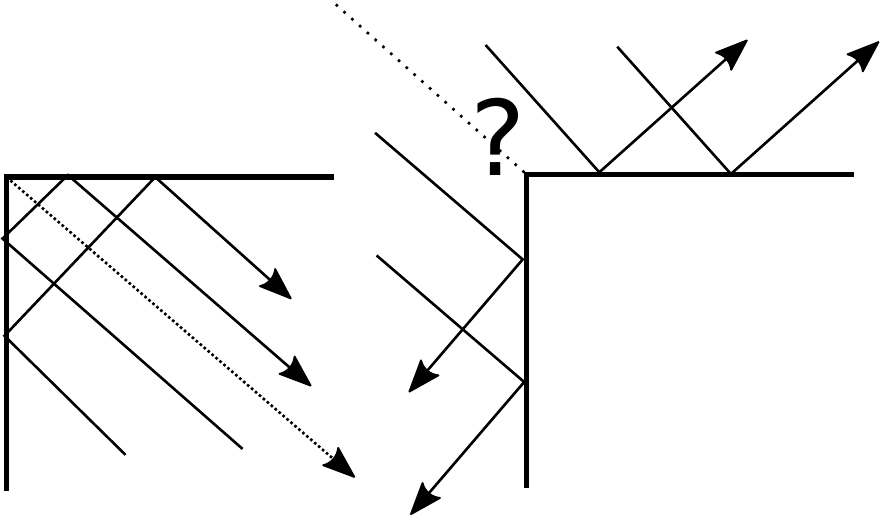}
		\caption{Отражение точки от вершины прямого угла (слева) и иллюстрация невозможности корректного продолжения траектории при отражении от вершины тупого угла (справа)} 
		\label{BoundRef}
	\end{figure}
	\begin{definition}
		Сложностью биллиарда назовем число $k$ --- число углов излома граничной кривой, равных $\frac{3\pi}{2}$.  
	\end{definition}
	
	\begin{definition}
		Кусочно-гладкое (некомпактное) 3-многообразие $$\tilde Q^3 = \{x\in M^4:|v(x)|= const\}$$ с краем назовем <<настоящей>> изоэнергетической 3-поверхностью данного биллиарда (см. главу 2.3).
	\end{definition}
	\begin{comment}
		Каждой точке биллиардной области $\varOmega$ соответствует окружность векторов скорости, за исключением вершин углов $\pi/2$ или $3\pi/2$. Подробное дока\-затель\-ство этого факта можно найти в [3].
	\end{comment}

	\section{Биллиардная область. Однородные биллиарды.}
	\subsection{Выбор биллиардной области.}
	Мы будем понимать под биллиардной областью $\varOmega$ односвязную область плоскости, ограни\-ченную дугами софокусных квадрик из семейства: 
	$$(b-\lambda)x^2+(a-\lambda)y^2=(a-\lambda)(b-\lambda),~\lambda \leq a. $$
	Здесь $\infty > a > b > 0$ --- фиксированная пара чисел (определяющая семейство софокусных
	квадрик), $\lambda$ --- параметр семейства (определяющий квадрику семейства).
	При $\lambda \in (0, a)$, где $\lambda \neq b$, --- это эллипсы или гиперболы. При
	$\lambda = b$ --- это объединение вырожденной гиперболы (образованной двумя горизонтальными лучами из фокусов) и вырожденного эллипса (отрезка между фокусами). Вертикальную прямую,
	соот\-вет\-ствующую параметру $\lambda = a$, мы будем называть гиперболой (а не вырожденной гиперболой), хоть по сути это вертикальный отрезок, но это не будет влиять на дальнейшие рассуждения. 
	Также потребуем, чтобы биллиард $\varOmega$ не содержал фокусов вне области или на ее границе.
	В дальнейшем под термином биллиард $\varOmega$ будет пониматься и динамическая система, и область одновременно. 
	\subsection{Интегрируемость.}
	Первым интегралом гамильтоновой системы называется постоянная вдоль траекторий функция. Мы ограничимся рассмотрением плоских биллиардов, ограниченных дугами софокусных квадрик, и интегрируемость таких биллиардов была замечена В.В. Козловым и Д.В. Трещевым [3]. Первый интеграл --- полная энергия, которая сохраняется в силу отсутствия трения в системе. Дополнительный интеграл --- это параметр каустики. Каустика --- это софокусная квадрика семейства, обладающая тем свойством, что если траектория касается каустики в одной точке, то она обязана касаться той же каустики после отражения от границы биллиарда. 
	Представим их формульную запись:
	$$ H = \frac{1}{2}(\dot{x}^2+\dot{y}^2)$$ (полная энергия)
	
	$$ \Lambda = \frac{\dot{x}^2}{a}+\frac{\dot{y}^2}{b} - \frac{(\dot{x}y-x\dot{y})^2}{ab}$$ (параметр каустики) 
	
	Здесь $a$ и $b$ --- параметры семейства софокусных квадрик (см. 2.1).

	\subsection{Пополнение изоэнергетической поверхности $Q^3$ особыми точками}
	Несмотря на невозможность определения корректного отражения в вершине угла $3\pi/2$, удобно рассматривать точки вида $(x_k,v) \in Q^3$, как настоящие, а не выколотые точки многообразия. Расслоим $\tilde Q^3$ на линии уровня функции $\varLambda = \alpha$ и пополним изоэнергетическую поверхность $\tilde Q^3$ следующим образом: на каждый уровень $\varLambda = \alpha$, где в вершину $x_k$ угла $3\pi/2$ попадает хотя бы одна траектория, добавим одну точку $(x_k,v_0)$, где $\varLambda(x_k, v_0) = \alpha$. Обозначим получившуюся поверхность через $Q^3$ и в дальнейшем будем изучать именно еe. 
	\begin{definition}
		Назовем естественной проекций $\pi$ многообразия $Q^3$ на биллиардную область $\varOmega$ гладкое отображение, действующее по правилу $(x,v) \rightarrow x$.	
	\end{definition}
	
	\begin{definition}
		Назовем особыми точками области (биллиарда) вершины углов $3\pi/2$, а особыми точками многообразия $Q^3$ --- прообразы при естественной проекции $\pi$ особых точек области. 
	\end{definition}
	\begin{comment}
	Вершины угла $\pi/2$ не будут считаться особыми точками биллиарда, так как в них задано корректно определённое по непрерывности отражение (см. рис \ref{BoundRef}, см. пункт 1.1). 
	\end{comment}	
	Рассмотрим интегрируемый биллиард $\varOmega$ c дополнительным интегралом $\varLambda$. Часто будет рассматриваться полный прообраз множества точек $A \subset \varOmega$ при естественной проекции $\pi$ и при фиксированных значениях интеграла $\varLambda(x,y,v_1,v_2) \in [c,d]$. Это множество задается как $$\pi^{-1}(A)|_{\varLambda \in [c,d]} = \{(x,v) \in Q^3 : x \in A, \varLambda(x,v) \in [c,d]\}.$$
	\subsection{Элементарные биллиарды.}
	Начнем с определения элементарного биллиарда --- биллиарда без особых точек, топо\-логия слоения Лиувилля для таких биллиардов была изучена В.В. Ведюшкиной в работе [\ref{V}]. 	
	\begin{definition}
		Рассмотрим компактный плоский (часть плоскости) односвязный бил\-лиард без особых точек (вершин угла $3\pi/2$), ограниченный дугами софокусных квадрик. Такой биллиард будем называть элементарным биллиардом.
	\end{definition}
	Приведем отношение эквивалентности на множестве элементарных биллиардов, данное в работах В.В. Ведюшкиной.
	\begin{definition}
		Элементарный биллиард $\varSigma$, ограниченный дугами квадрик из софокусного семейства, называется эквивалентным другому элементарному биллиарду $\varSigma'$, ограниченному дугами квадрик из того же семейства, если $\varSigma$ можно получить из $\varSigma'$ путем композиции перечисленных ниже преобразований.
		\begin{enumerate}
			\item Последовательное изменение сегментов границы путем непрерывной деформации в классе квадрик. Потребуем, чтобы во все время деформации сегмент границы лежал либо на софокусном эллипсе (т.е. их параметр квадрики, на которой он лежит, меняется в пределах $(-\infty, b)$), либо на софокусной гиперболе (т.е. параметры квадрики, на которой он лежит, меняется в
			пределах $(b, a]$), либо является отрезком фокальной прямой (т.е. параметр софокусной квадрики, на которой он лежит, равен $b$ во все время деформации);
			\item Симметрия относительно оси семейства;
			\item Объединение нескольких простейших элементарных биллиардов в один или же путем разбиения одного биллиарда на более мелкие (объединение и разбиение как подмножеств $\mathbb{R}^2$).
		\end{enumerate}
		
	\end{definition}
	Также В.В. Ведюшкиной элементарные биллиарды были классифицированы с точностью до указанного выше отношения эквивалентности. 
	\begin{statement}[В.В. Ведюшкина [\ref{V}]
	Любой односвязный элементарный биллиард $\varSigma$ эквивалентен биллиарду, принадлежащему одному из следующих двух серий (все они представлены на рисунках ):
		\begin{enumerate}
		\item Односвязные элементарные биллиарды содержащие
		отрезок фокальной прямой между фокусами (внутри биллиарда или на границе). Существует ровно шесть типов. Представители всех классов этой серии изображены на рисунке \ref{A}. Для каждого класса укажем $f$ --- количество фокусов, принадлежащих области, $f'$
		--- число фокусов принадлежащих границе биллиарда. Такие биллиарды будем обозначать $A_f$, если их граница не содержит отрезок фокальной прямой, и $A'_f$ иначе.
		\item Односвязные элементарные 	биллиарды не содержащие отрезка фокальной прямой между фокусами. Каждый такой биллиард ограничивает четырёхугольник, состоящий из дуг двух эллипсов и двух гипербол (быть может совпадающих). Существует ровно шесть типов. Представители всех классов этой серии изображены на рисунке \ref{B}. Такие биллиарды будем обозначать либо $B_n$, либо $B'_n$, либо $B''_n$ в зависимости от того, ноль, один или два отрезка границы лежат на фокальной прямой, где $n$ --- это количество связных компонент фокальной прямой в биллиарде. Будем называть их биллиардами типа $B$. 
		\end{enumerate}
		При этом области, принадлежащие к различным сериям ($A$ или $B$) неэквивалентны
		между собой, а также неэквивалентны между собой внутри каждой серии области с различными индексами.
	
	\end{statement}
	 \begin{figure}[!ht]
	\centering
	\includegraphics[width=0.7\textwidth]{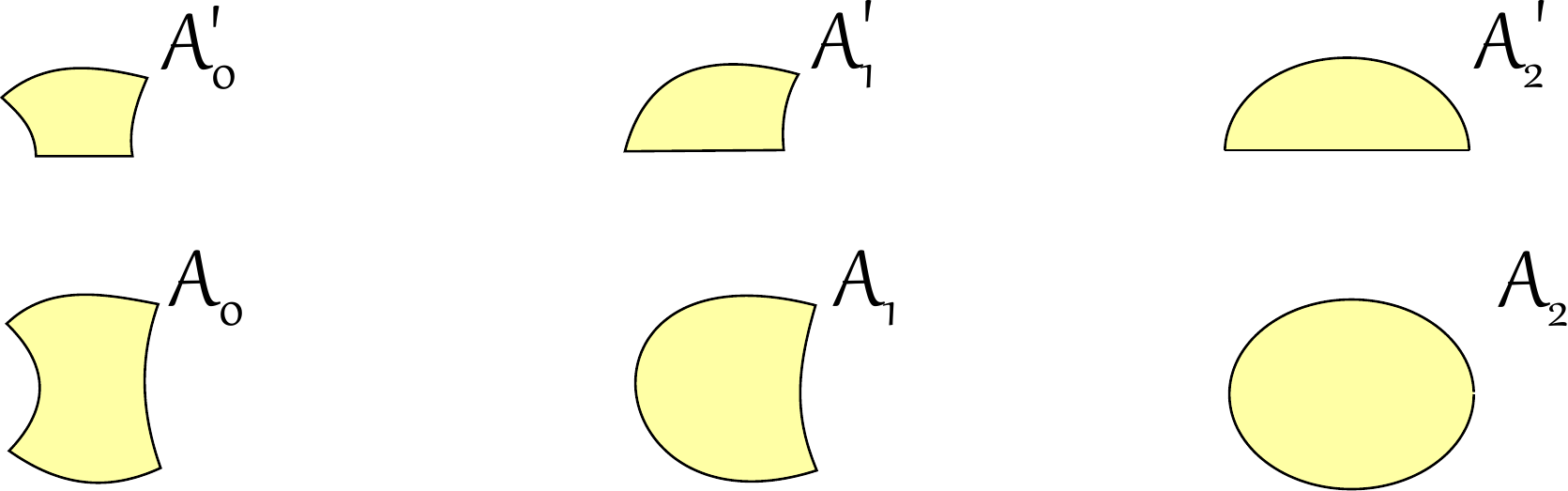}
	\caption{Элементарные биллиарды серии $A$.} 
	\label{A}
\end{figure}
	 \begin{figure}[!ht]
	\centering
	\includegraphics[width=0.7\textwidth]{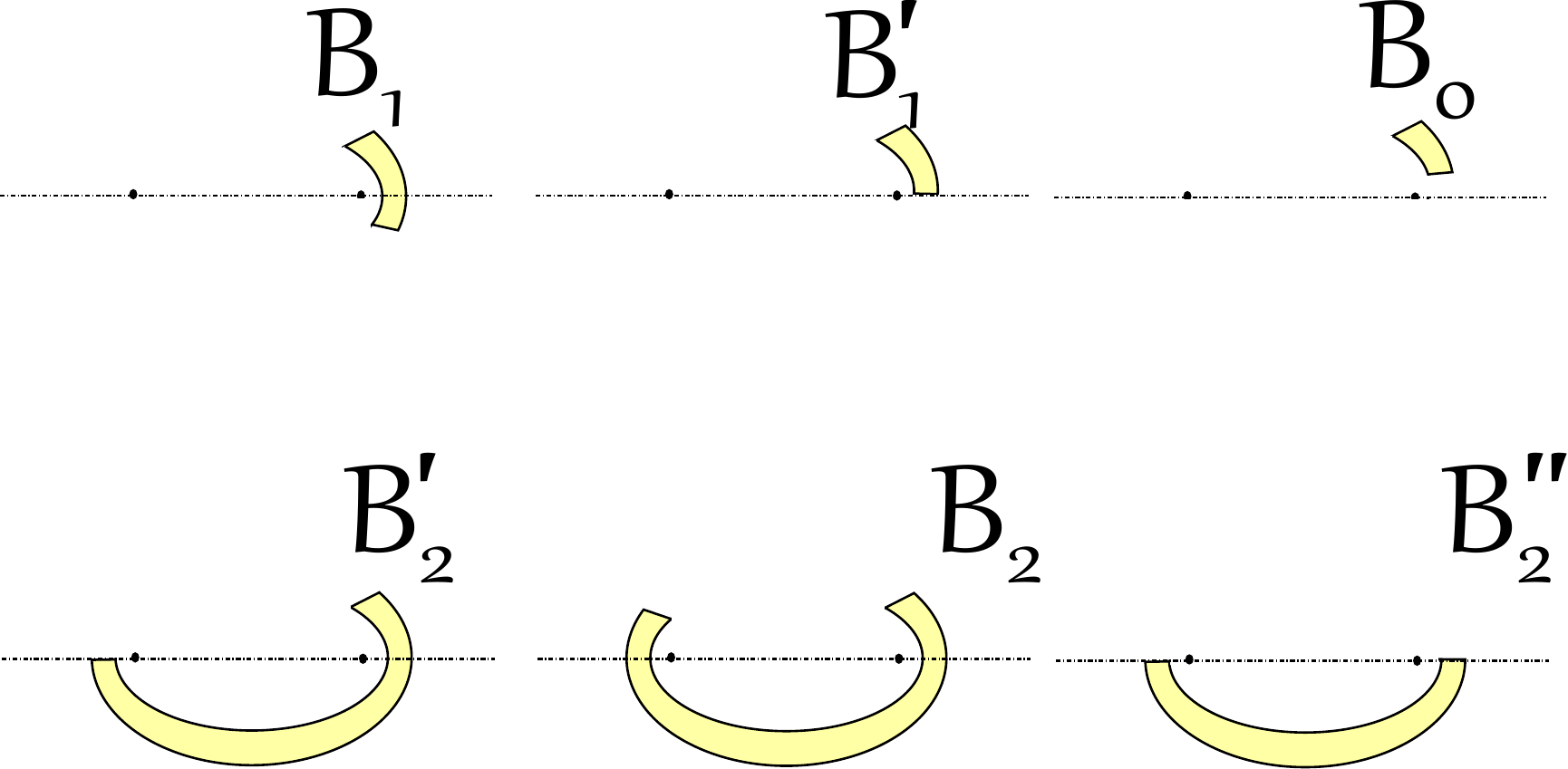}
	\caption{Элементарные биллиарды серии $B$.} 
	\label{B}
\end{figure}
	\subsection{Однородные биллиарды.}
	Наличие внутри биллиарда сегментов фокальной прямой одновременно вне и внутри фокусов значительно усложняет топологию бифуркационных уровней. Поэтому мы сузим класс рассматриваемых биллиардов до однородных биллиардов.
	\begin{definition}\label{DefHom}
		Рассмотрим плоский (не обязательно элементарный) биллиард $\varOmega$ произ\-вольной сложности $k$. Если биллиард $\varOmega$ не содержит участка фокальной прямой между фокусами (сегментов вырожденного эллипса), то назовем его однородно-эллиптическим. Если любой связный сегмент фокальной прямой, содержащийся в биллиарде $\varOmega$, лежит между фокусами, то назовем его однородно-гипер\-бо\-личес\-ким.  В ином случае, будем называть его неод\-нородным. Биллиард $\varOmega$, который не пересекает фокальную прямую, будем считать и однородно-гиперболическим, и однородно-эллиптическим (см. рис. \ref{HomNonHom}).
	\end{definition}

	\begin{comment}
	Однородно-эллиптическими элементарными биллиардами являются все биллиарды серии $B$, однородно-гиперболическими элементарными биллиардами --- биллиарды $A_0$, $A'_0$ и $B_0$. Остальные элементарные биллиарды неоднородны (элементарный биллиард $B_0$ одновременно однородно-эллиптический и однородно-гиперболический, так как не пере\-секает фокальную прямую). 
	\end{comment}

		\begin{figure}[!ht]
		\centering
		\includegraphics[width=0.6\textwidth]{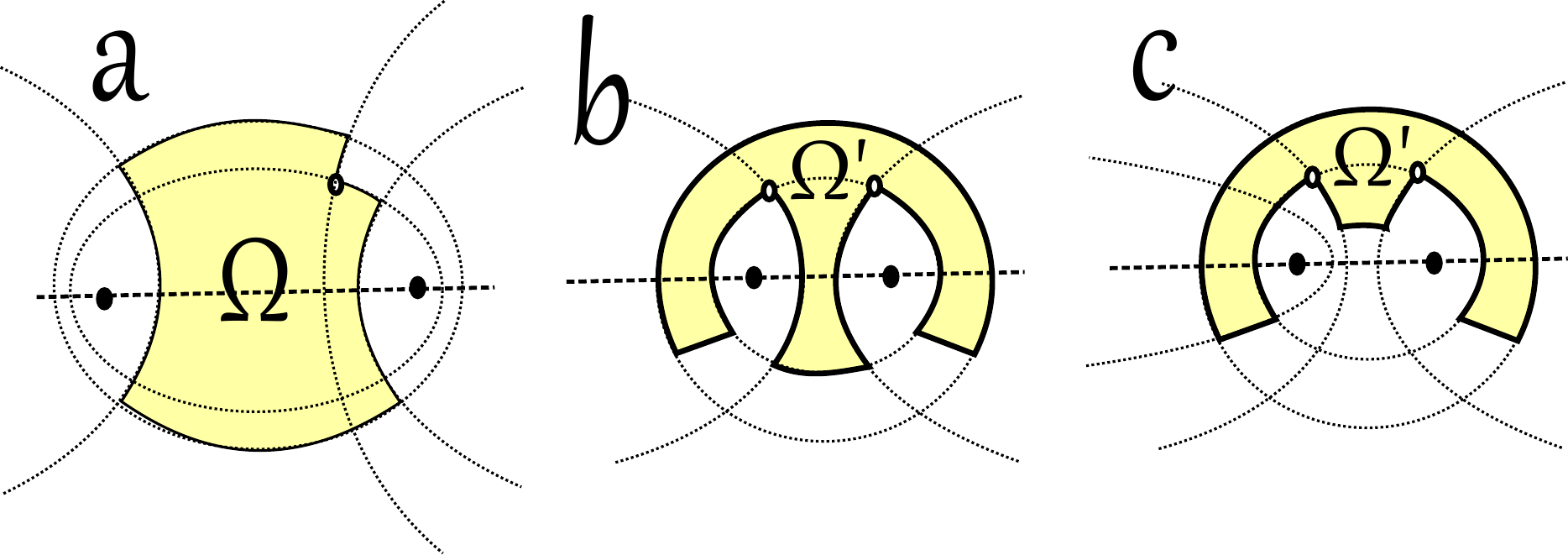}
		\caption{а --- пример однородно-гиперболического биллиарда, б --- пример неоднородного биллиарда, с --- пример однородно-эллиптического биллиарда}
		\label{HomNonHom}
	\end{figure}

	\section{Известные результаты теории биллиардов.}
	\subsection{Описание топологии слоений Лиувилля для элементарных биллиар\-дов.}
	
	Начнем с определения критических  и регулярных значений дополнительного интеграла $\varLambda$. 
	
	Зафиксируем $H = const$ и ограничим наше многообразие $M^4$ на изоэнергетическую поверхность $Q^3$. Тогда, в случае элементарных биллиардов у слоения Лиувилля было три особых уровня: $\varLambda = 0$, $\varLambda = b$ и $\varLambda = a$. Топология $Q^3$ в элементарных биллиардах менялась только при переходе дополнительного интеграла через эти значения. А в теории биллиардов с невыпуклыми углами к ним также добавятся $2k$ новых особых значений дополнительного интеграла $\varLambda$ при переходе через которые топология $Q^3$ изменится, где $k$ --- сложность биллиарда (см. опр. 1), а именно, число углов излома граничной кривой биллиарда, равных $3\pi/2$.
	 
	\begin{definition}
	Фиксируем биллиард $\varOmega$. Граница области $\varOmega$ образована сегментами эллипсов и гипербол семейства. Если эти сегменты границы не содержат особых точек, то обозначим через $ $min$_i$ и $ $max$_j$ значения параметра $\varLambda$, которым соответствуют эллиптические и гиперболические сегменты границы области $\varOmega$ такого типа. Если область $\varOmega$ имеет непустое пересечение с прямой Oy, то дополним набор $ $max$_j$ значением $a$. Значение параметра $\varLambda$ сегментов границы, на которых лежат особые точки, обозначим через $\lambda_I$. 

	Назовем особыми следующие значения интеграла $\varLambda$: 
	
	\begin{enumerate}
	\item (Локально) минимальные значения интеграла $\varLambda = ~$min$_i$; 
	
	\item Седловое значение интеграла $\varLambda = b$;
	
	\item (Локально) максимальные значения интеграла $\varLambda = ~$max$_j$; 
	
	\item Значения $\varLambda = \lambda_i$, где $1 \leq i \leq 2k$ (так как каждая особая точка лежит либо на пересечении эллипса и гиперболы семейства, либо гиперболы и вырожденного эллипса).
	\end{enumerate}
	\end{definition}

	B. B. Ведюшкиной [\ref{V}] были построены инварианты Фоменко-Цишанга для элементарных биллиардов, приведем только часть ее результата --- описание грубых молекул. 
	\begin{theorem}[В.В. Ведюшкина]\label{thVed}
		Прообраз $b-\varepsilon \leq \varLambda \leq b+\varepsilon$ в изоэнергетической поверхности $Q^3$ элементарного биллиарда $\varSigma$ при достаточно малом значении $\varepsilon$ гомеоморфен следующим трёхмерным многообразиям (перечисленные ниже атомы трехмерные):
		\begin{enumerate}
			\item атом B для областей $A_2$, $A_0$;
			\item атом $A^{*}$
			для области $A_1$;
			\item атом $B_n$ для $B_n$, 
			$B_{n+1}^{'}$, где $n > 0$;
			\item произведение тора на отрезок для областей $A'_2$, $A'_1$, $A'_0$, $B_0$, $B'_1$, $B''_2$;
		\end{enumerate}
	\end{theorem}
	\subsection{Описание регулярных слоев невыпуклых биллиардов.}
	Предоставим описание регулярных слоёв интеграла для биллиардов с $k \neq 0$. Ранее описание регулярных слоёв биллиардов произвольной сложности возникало в работах В. Драгович и М. Раднович [4-6], эти работы бази\-руются на анализе динамики интегральных траекторий. 
	
	\begin{theorem}[В. Драгович, М. Раднович]\label{thDR}
	Для всех неособых значений интеграла повер\-хность уровня интеграла $\varLambda$
	в изо\-энер\-ге\-тичес\-кой поверхности $Q^3$ биллиарда $\varOmega$ гомео\-морфна объединению поверхностей рода $(k'+1)$, где $k'$ --- количество особых точек вне интегрального эллипса, если $\varLambda<b$, или внутри интег\-ральной гиперболы, если $\varLambda>b$ (в <<настоящей>> изоэнергетической поверхности $\tilde Q^3$ на каждом уровне будет выколото $k'$ точек).
	\end{theorem}

	Любая регулярная поверхность уровня для произвольного биллиарда $\varOmega$ будет гомео\-морфна сфере с некоторым количеством ручек и проколов (будем обозначать такую сферу через $G^{(k'+1)}_2$) и как видно из теоремы \ref{thDR}, род этой поверхности меняется при изменении $k'$ --- числа особых точек в области возможного движения. Таким образом, сформулируем ключевую идею изучения топологии слоения Лиувилля для невыпуклых биллиардов --- на критических значениях интеграла ($\Lambda = b$ и $\Lambda = \lambda_i$) меняется $k'$ и, как следствие теоремы \ref{thDR}, меняется топология совместных поверхностей уровней $\varLambda = \alpha$ (см. рис. \ref{RegFib}, рис. \ref{SingFib}).   
	\begin{figure}[!ht]
		\centering
		\includegraphics[width=0.5\textwidth]{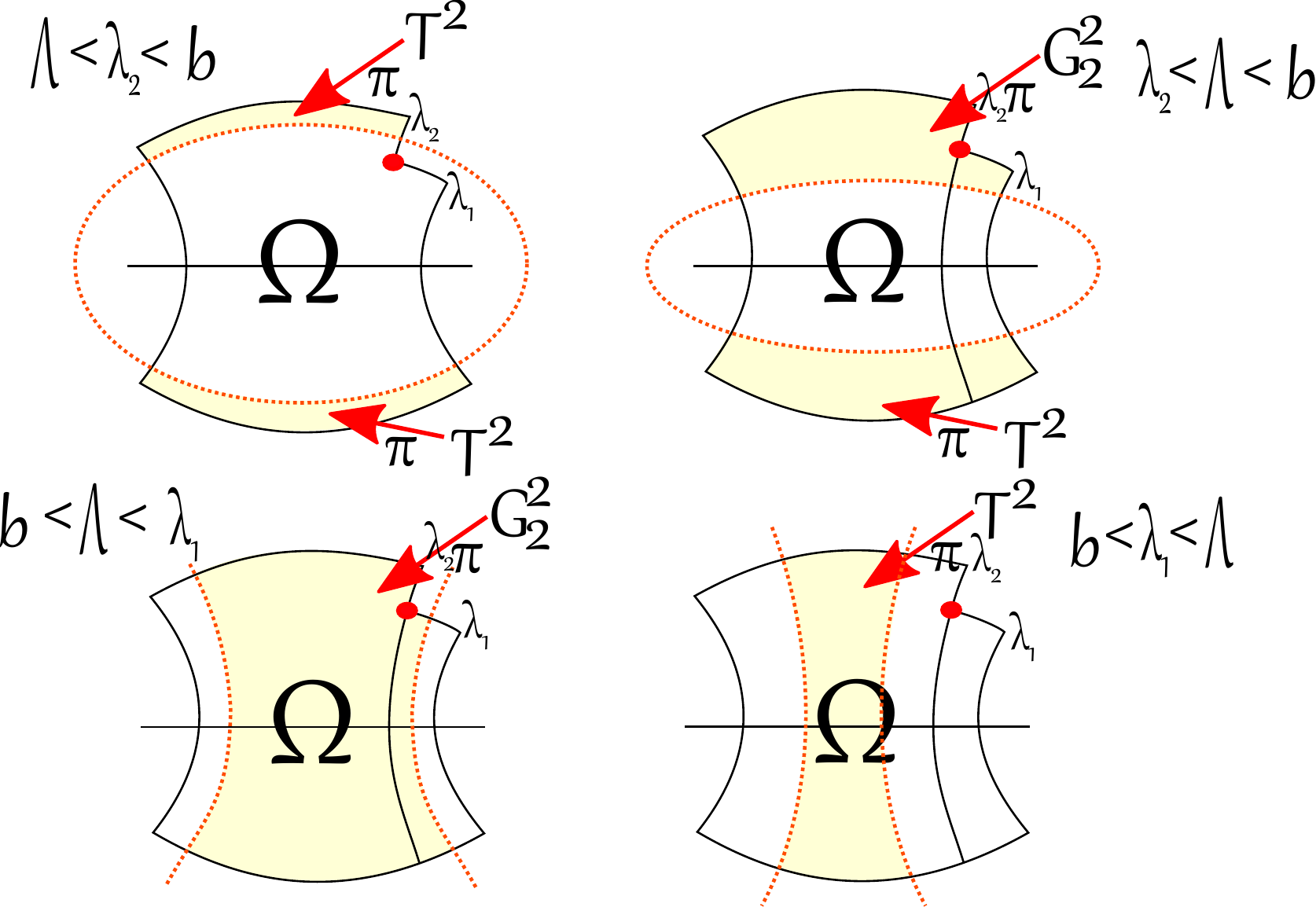}
		\caption{Области возможного движения при различных регулярных значениях интеграла и их прообразы в $Q^3$.}
			\label{RegFib}
	\end{figure}
	\begin{figure}[!ht]
		\centering
		\includegraphics[width=0.5\textwidth]{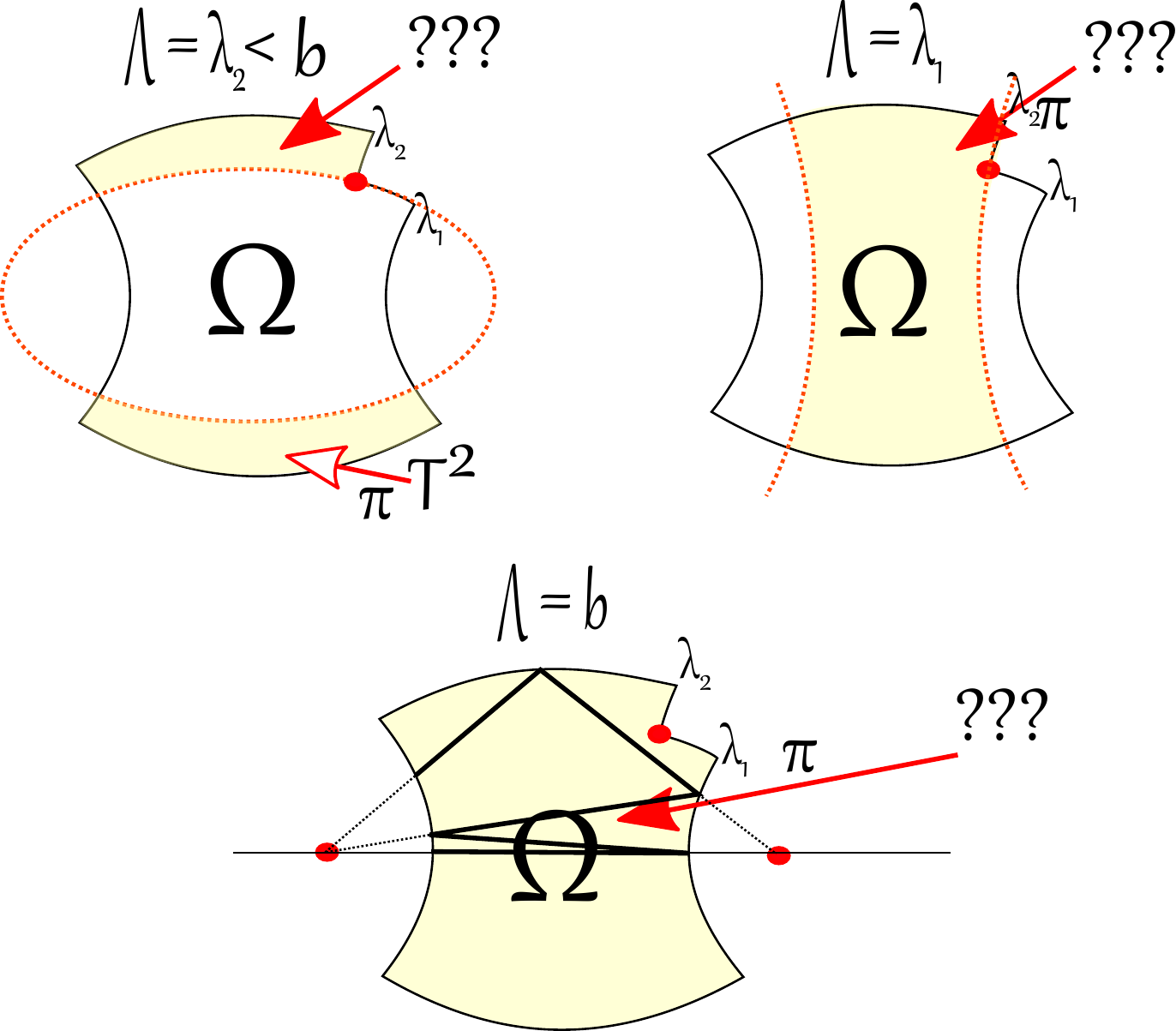}
		\caption{Области возможного движения при критических значениях интеграла.}
		\label{SingFib}
	\end{figure}
\section{Разбиение однородного биллиарда на элементарные.}
\subsection{Разбиение однородного биллиарда}
Главной идеей изучения топологии многообразия $Q^3$ будет разрезание этого многообразия на более мелкие и простые части. Для этого мы сначала разобьем биллиард $\varOmega$ на элементарные биллиарды $\varSigma_1, \ldots, \varSigma_N$ специальным образом, а потом <<поднимем>> это разбиение на уровень многообразия $Q^3$ с помощью естественной проекции $\pi: Q^3 \rightarrow \varOmega$, где $\pi(x,v) = x$. Так как топология окрестности особого слоя для элементарных биллиардов уже изучена, то мы сможем использовать это для построения топологии $Q^3$.    

Представим способ выбора разбиения на однородном биллиарде: будем разрезать по одному типу квадрик, которые проходят через особые точки. 
\begin{definition}\label{PartH}
	Рассмотрим одно\-родно-гиперболический биллиард $\varOmega$ произвольной слож\-ности $k$. Будем выбирать на нем разбиение $\varSigma_1, \ldots, \varSigma_N$ следующим образом: 
	проведем все, в том числе и вырожденную, гиперболы на которых лежат особые точки (вершины углов $3\pi/2$).
	
	В результате биллиард $\varOmega$ разобьётся на не более чем $2k$ элементарных биллиардов $\varSigma$ (несколько особых точек могут лежать на одной гиперболе (эллипсе). 
	
	Назовем такое разбиение однородного биллиарда разбиением $\varSigma_1, \ldots, \varSigma_N$, числом $N$ будем обозначать количество элементов в нем, а через $n$ обозначим количество гипербол $\lambda_i$ по которым совершались разрезы.
\end{definition}
\begin{definition}\label{PartE}
	Рассмотрим одно\-родно-эллиптический биллиард $\varOmega$ произвольной слож\-ности $k$. Будем выбирать на нем разбиение $\varSigma_1, \ldots, \varSigma_N$ следующим образом: 
	проведем все, в том числе и вырожденный, эллипсы на которых лежат особые точки (вершины углов $3\pi/2$).
	
	В результате биллиард $\varOmega$ разобьётся на не более чем $2k$ элементарных биллиардов $\varSigma$ (несколько особых точек может лежать на одном эллипсе). 
	
	Назовем такое разбиение однородного биллиарда разбиением $\varSigma_1, \ldots, \varSigma_N$, числом $N$ будем обозначать количество элементов в нем, а через $n$ обозначим количество эллипсов $\lambda_i$ по которым совершались разрезы.
\end{definition}

В обоих определениях $N > n$, так как разрез по каждой квадрике добавляет хотя бы один новый элементарный биллиард $\varSigma_j$ в разбиение.
\subsection{Области возможного движения}
Выше мы выбрали разбиение всего биллиарда $\varOmega$ на элементарные биллиарды. Однако, биллиардное движение не обязательно происходит внутри целого биллиарда, т.е. при некоторых значениях дополнительного интеграла $\varLambda = \alpha$ некоторые точки $x \in \varOmega$ могут не оснащаться ни одним вектором скорости. Подробнее, фиксируем значение дополнительного интеграла $\varLambda = \alpha$ и рассмотрим такие точки $x \in \varOmega$ биллиарда $\varOmega$, в которых существует такой вектор скорости $v$, чтобы пара $(x,v)$ лежала в $Q^3$ и $\varLambda(x,v) = \alpha$. В силу теоремы Якоби-Шаля [1] в биллиардах существует три типа областей возможного движения в зависимости от $\alpha$: 
\begin{enumerate}
	\item Если $\alpha < b$, то область возможного движения заполняет биллиард $\varOmega$ вне эллипса семейства с параметром $\lambda = \alpha$;
	\item Если $\alpha = b$, то область возможного движения заполняет весь биллиард $\varOmega$;
	\item Если $\alpha > b$, то область возможного движения заполняет биллиард $\varOmega$ внутри гиперболы семейства с параметром $\lambda = \alpha$.
\end{enumerate}
\subsection{Разбиение областей возможного движения}
Теперь разобьем не только целый однородный биллиард $\varOmega$ на элементарные биллиарды, но и все возможные области возможного движения, возникающие в биллиарде $\varOmega$ при фиксировании значений дополнительного интеграла $\varLambda = \alpha$. Очевидно, что в таком случае разрезы будут проводиться по некоторому подмножеству квадрик $\lambda_{i_1}, \ldots, \lambda_{i_n} \subseteq \lambda_1, \ldots, \lambda_n$, так как некоторые квадрики $\lambda_1, \ldots, \lambda_n$ могут оказаться вне области возможного движения. Таким образом, зафиксировав разбиение однородного биллиарда $\varOmega$ на элементарные биллиарды $\varSigma_1, \ldots, \varSigma_N$, мы фик\-си\-руем также и разбиение всех возможных областей движения на элемен\-тарные биллиарды $\varSigma'_1, \ldots, \varSigma'_{N'}$, где каждый элемен\-тарный биллиард $\varSigma'_1, \ldots, \varSigma'_{N'}$ эквивалентен некоторому биллиарду из множества $\varSigma_1, \ldots, \varSigma_N$ (см. рис. \ref{PosMotion}).

\begin{figure}[!ht]
	\centering
	\includegraphics[width=0.5\textwidth]{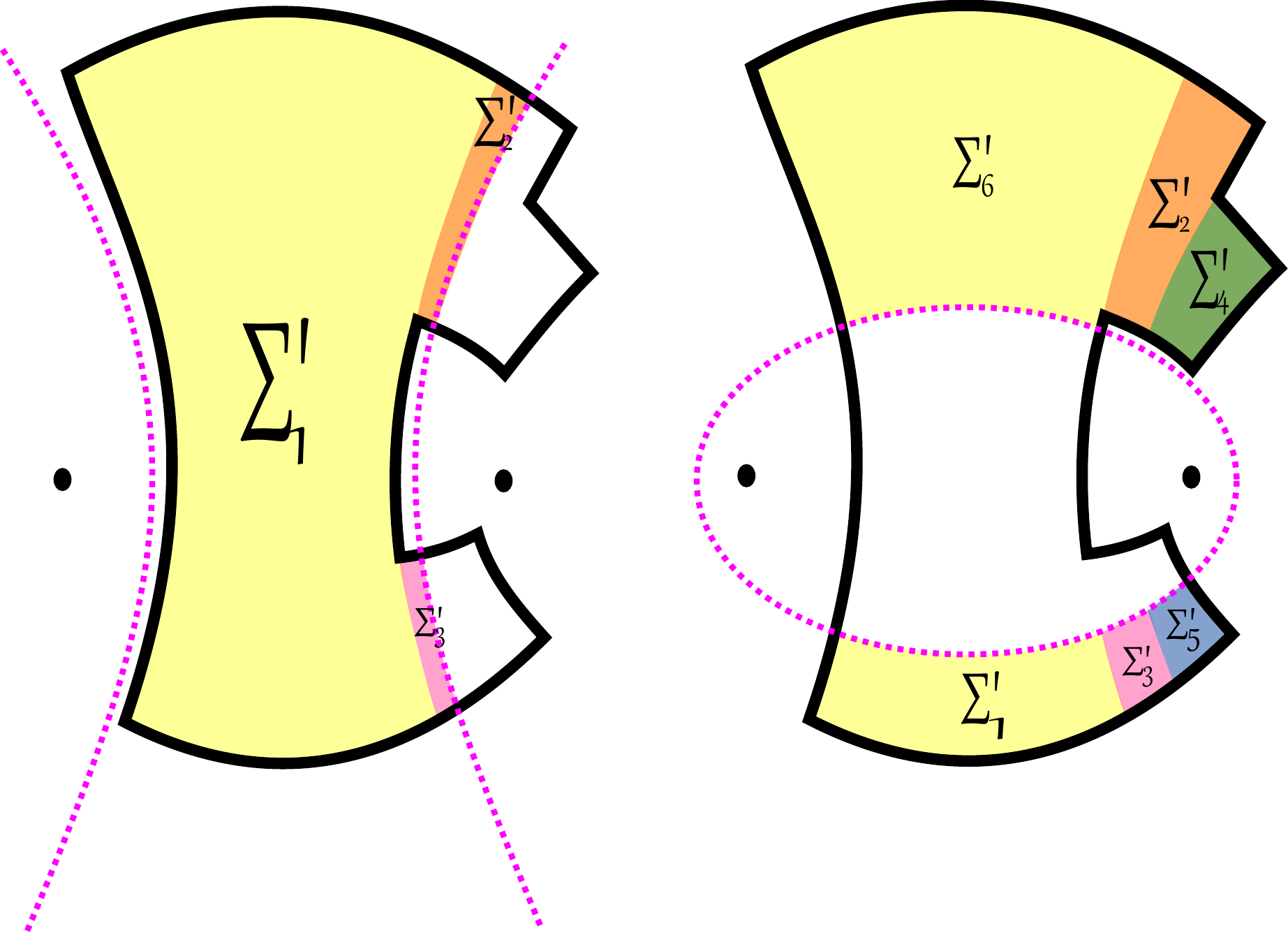}
	\caption{Фиолетовым изображена каустика, белые участки биллиарда не входят в область возможного движения. Изображен пример разбиения области возможного движения при эллиптических значениях интеграла (справа) и при гиперболических значениях интеграла (слева).} 
	\label{PosMotion}
\end{figure}
Теперь уточним определение граничной дуги гиперболы $\lambda_i$. Сначала дадим определение для однородно-гиперболического биллиарда, потом --- для однородно-эллиптического. 
\begin{definition}
	Рассмотрим однородно-гиперболический биллиард $\varOmega$ с выбранным на нем разбиением $\varSigma_1, \ldots, \varSigma_{N}$. Назовем граничной дугой гиперболы $\lambda_i$ объединение всех сегментов границ элементарных биллиардов $\varSigma_j$ лежащих на гиперболе с параметром $\lambda_i$ (см. рис. \ref{QuadBound})  
\end{definition}
\begin{definition}
	Рассмотрим однородно-эллиптический биллиард $\varOmega$ с выбранным на нем разбиением $\varSigma_1, \ldots, \varSigma_{N}$. Назовем граничной дугой эллипса $\lambda_i$ объединение всех сегментов границ элементарных биллиардов $\varSigma_j$ лежащих на эллипсе с параметром $\lambda_i$. 
\end{definition}
\begin{figure}[!ht]
	\centering
	\includegraphics[width=0.5\textwidth]{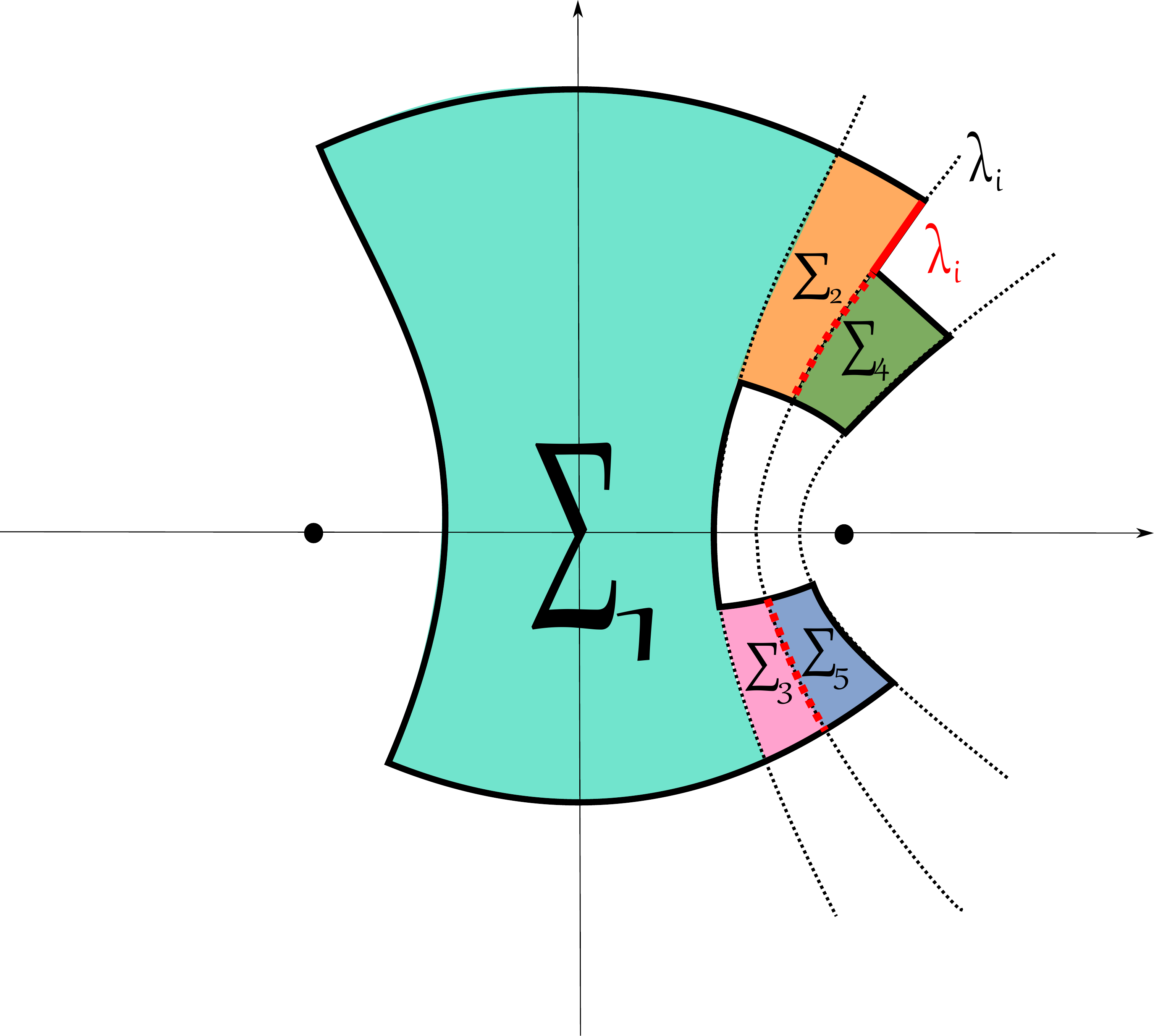}
	\caption{Пример однородно-гиперболического биллиарда $\varOmega$ c выбранным на нем разбиением. Черным пунктиром обозначена гипербола с параметром $\lambda_i$, а красным --- дуга граничной гиперболы $\lambda_i$.} 
	\label{QuadBound}
\end{figure}
Отметим, что дуги граничной квадрики определённые подобным образом действительно могут быть несвязными, так как некоторые границы между элементарными биллиардами разбиения --- внутренние. Однако, для будущего анализа нам проще рассматривать границы между элементарными биллиардами разбиения как настоящие границы, на некоторых сегментах которых отменен биллиардный закон. В дальнейшем под $\lambda_i$ будут подразумеваться граничные дуги гиперболы с параметром, также обозначаемым $\lambda_i$.
\subsection{Отношение эквивалентности на невыпуклых биллиардах.}
Зададим отношение эквивалентности на множестве однородных биллиардов. 
Теперь рассмотрим два однородных биллиардов $\varOmega$ и $\varOmega'$ c выбранными на них разбиениями на элементарные биллиарды $\varSigma_1, \ldots, \varSigma_N$ и $\varSigma'_1, \ldots, \varSigma'_N$ соответственно. 
\begin{definition}{\label{Eq}}
	Однородный биллиард $\varOmega$, ограниченный дугами квадрик из софокус\-ного семейства, называется эквивалентным другому однородному биллиарду $\varOmega'$, огра\-ничен\-ному дугами квадрик из того же семейства, если:
	\begin{enumerate}
		\item между множествами элементарных биллиардов $\varSigma_1, \ldots, \varSigma_N$ и $\varSigma'_1, \ldots, \varSigma'_N$ существует биекция $E$, причем $E(\varSigma_j) \sim \varSigma_j$ в смысле предыдущего определения;
		\item если элементарные биллиарды $\varSigma_j$ и $\varSigma_i$ пересекались по сегменту граничной дуги $\lambda_{i,t}$ и на дуге $\lambda_{i,t}$ лежало $t$ особых точек, то элементарные биллиарды $E(\varSigma_j)$ и $E(\varSigma_i)$ вновь пересекутся по дуге граничной гиперболы с таким же количеством особых точек $t$ на ней (см. рис. \ref{Eqnoteq}).  
	\end{enumerate}
	
\end{definition}
	\begin{figure}[!ht]
	\centering
	\includegraphics[width=0.5\textwidth]{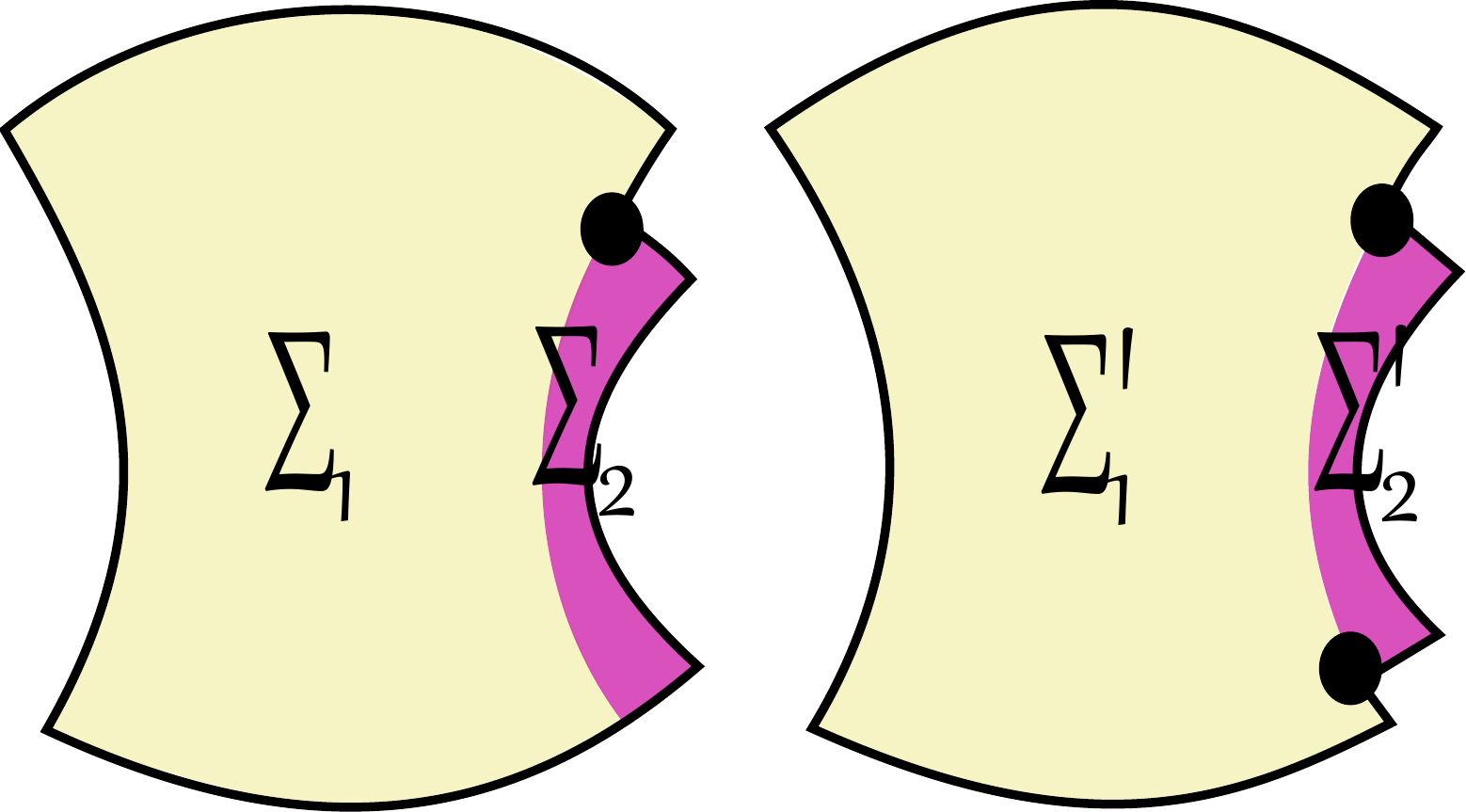}
	\caption{Пример двух неэквивалентных в смысле определения \ref{Eq} биллиардов. Нарушается условие 2: $\varSigma_1 \sim \varSigma'_1$ и $\varSigma_2 \sim \varSigma'_2$ как элементарные биллиарды, однако, условие два, очевидно, не выполняется. Отметим, что в силу Теоремы \ref{thDR} слоение Лиувилля у правого и левого биллиарда будут отличаться.}
	\label{Eqnoteq}
\end{figure}
В дальнейшем все биллиарды предполагаются однородными, рассматриваемыми с точностью до эквивалентности в смысле определения \ref{DefHom}. 
\section{Определение 3-атома для однородных невыпуклых биллиардов.}
Для исследования биллиардов с невыпуклыми углами на границе области следует расширить определение атома, так как на совместных поверхностях уровня интеграла могут лежать выколотые точки, а, как следствие, окрестности поверхностей уровня интегралов удобнее рассматривать как клеточные комплексы. Дадим определение для атома на критическом, не минимаксном, уровне $c$ ($c = b$ или $c = \lambda_i$ для некоторого $i$). 
\begin{definition}
	Трехмерным атомом (3-атомом) назовем трехмерную окрестность $U \subset Q^3$ двумерного особого слоя $G$, задаваемую неравенством $$c - \varepsilon \leq \Lambda \leq c + \varepsilon$$ для достаточно малого $\varepsilon$, расслоённую на двумерные поверхности уровни функции $\varLambda$ и рассматриваемую с точностью до послойной эквивалентности. Трехмерный атом для окрестности особого уровня $\varLambda = b$ обозначим через $U$, для $\lambda = \lambda_i$ --- через $U_{\lambda_i}$.  
\end{definition}
Также обозначим через $U_{(\lambda_i)}$ окрестность произвольного особого слоя. Рассмотрим однородный биллиард $\varOmega$ с выбранным на нем разбиением $\varSigma_1, \ldots, \varSigma_{N}$ c граничными гиперболами (эллипсами) $\lambda_1, \ldots, \lambda_n$. Также обозначим через $\varOmega^{+}_{c}$ --- область возможного движения при фиксированном дополнительном интеграле $\varLambda = c + \varepsilon$, через $\varOmega^{-}_{c}$ ---  при фиксированном дополнительном интеграле $\varLambda = c - \varepsilon$, а через $\varOmega_{c}$ --- область возможного движения при фиксированном дополнительном интеграле $\varLambda = c$. Заметим, что в силу определения биллиардной каустики (см. 2.2), граница областей возможного движения на всех уровнях интеграла, кроме уровня $\varLambda = b$, содержит сегменты каустики.
\begin{enumerate}
	\item Выберем 0-клетки. Рассмотрим биллиард $\varOmega$ и области возможного движения $\varOmega^{-}_c$, $\varOmega_c$, $\varOmega^{+}_c$, зафиксировав дополнительный интеграл $\varLambda$ равным $c-\varepsilon$, $c$ или $c + \varepsilon$ соответственно. На всех трех областях возможного движения проведем: все граничные квадрики $\lambda_1, \ldots, \lambda_n$ из разбиения $\varSigma_1, \ldots, \varSigma_{N}$, а в случае $c = b$ дополнительно проведем фокальную прямую. Рассмотрим пересечения проведенных квадрик с границами областей возможного движения  $\varOmega^{-}_c$, $\varOmega_c$, $\varOmega^{+}_c$. Обозначим данное множество точек через $P$. 0-клетками трехмерных комплексов $U_{(\lambda_i)}$ будут прообразы при естественной проекции $\pi$ выбранных вершин пересечений при заданных значениях дополнительного интеграла $\varLambda$ (см. пример на рис. \ref{0cell}). 
	
	\begin{figure}[h!]
		\begin{center}
			\begin{minipage}[h]{1.0\linewidth}
				\includegraphics[width=1\linewidth]{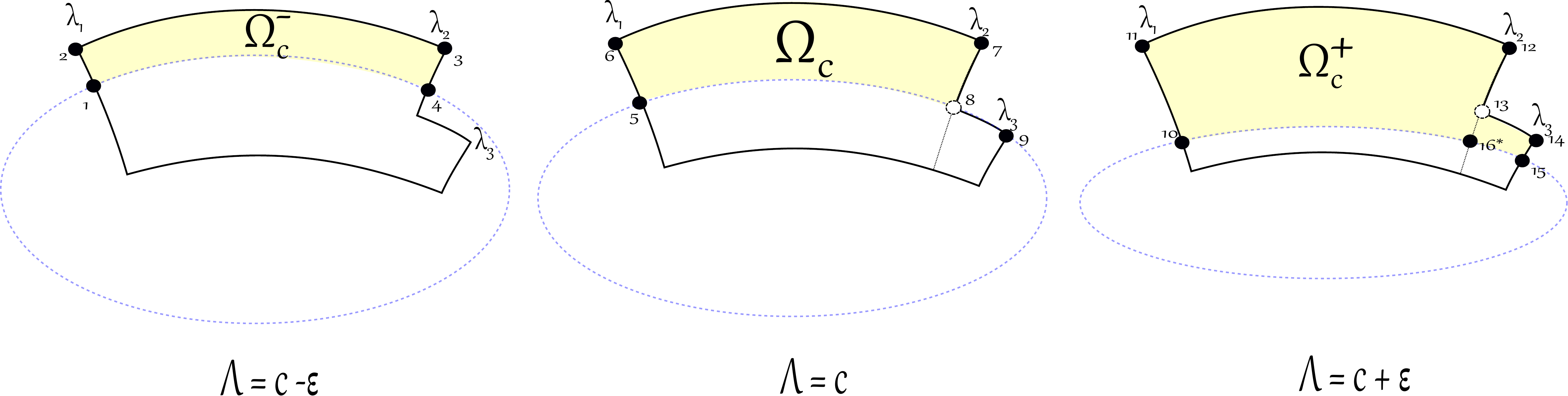}
				\caption{Пример 0-клеток комплекса $U_{\lambda_2}$ для биллиарда, изображённого на рисунке с зафиксированными значениями дополнительного интеграла $\varLambda = \lambda_2 - \varepsilon$ (слева), $\varLambda = \lambda_2$ (в центре) и $\varLambda = \lambda_2 + \varepsilon$ (справа). Области возможного движения обозначены желтым. Всего в комплексе $U_{\lambda_2}$ $17$ 0-клеток, а множество $P$ состоит $16$ элементов. Если номер точки из $P$ помечен звездочкой, то в данную точку проецируются естественной проекцией $\pi$ две 0-клетки из $U_{\lambda_2}$, иначе --- одна. Белым отмечены особые точки биллиарда $\varOmega$.} 
				\label{0cell} 
			\end{minipage}
		\end{center}
	\end{figure}
	\item 1-клетки комплекса $U_{(\lambda_i)}$ будут двух разных типов, то есть $U^{1}_{(\lambda_i)} = U^{1,1}_{(\lambda_i)} \cup U^{1,2}_{(\lambda_i)}$. Первый тип 1-клеток комплекса $U_{(\lambda_i)}$ --- это связные в $U_{(\lambda_i)}$ прообразы при естественной проекции $\pi$ точек из множества $P$ при всех значениях дополнительного интеграла $\varLambda$ из множества $(c-\varepsilon, c) \cap (c, c + \varepsilon)$. Связные прообразы при естественной проекции $\pi$ --- это в точности такие пары $(x,v) \in U_{(\lambda_i)}$, где векторы скорости $v$ направлены в одном направлении (от правого фокуса, от левого фокуса, к правому фокусу или к левому фокусу) на каждом уровне интеграла $\varLambda$. 
	Построим теперь второй тип 1-клеток, рассмотрим все ограниченные точками из $P$ связные сегменты границ областей возможного движения $\varOmega^{+}_c$, $\varOmega_c$ и $\varOmega^{-}_c$ (эллиптические \textbf{и} гиперболические) и сегменты квадрик $\lambda_1, \ldots, \lambda_n$, обозначим множество таких сегментов через $Q$. Также, в случае седлового атома $\varLambda = b$, добавим в $Q$, если таковой имеется, сегмент фокальной прямой между точками из $P$.  Второй тип 1-клеток --- это все связные прообразы сегментов квадрик из $Q$ при естественной проекции $\pi$, при фиксированных значениях интеграла $\varLambda = c - \varepsilon$, $\varLambda = c$ и $\varLambda = c + \varepsilon$ (см. пример на рис. \ref{1cell}).
		\begin{figure}[h!]
		\begin{center}
			\begin{minipage}[h]{1.0\linewidth}
				\includegraphics[width=1\linewidth]{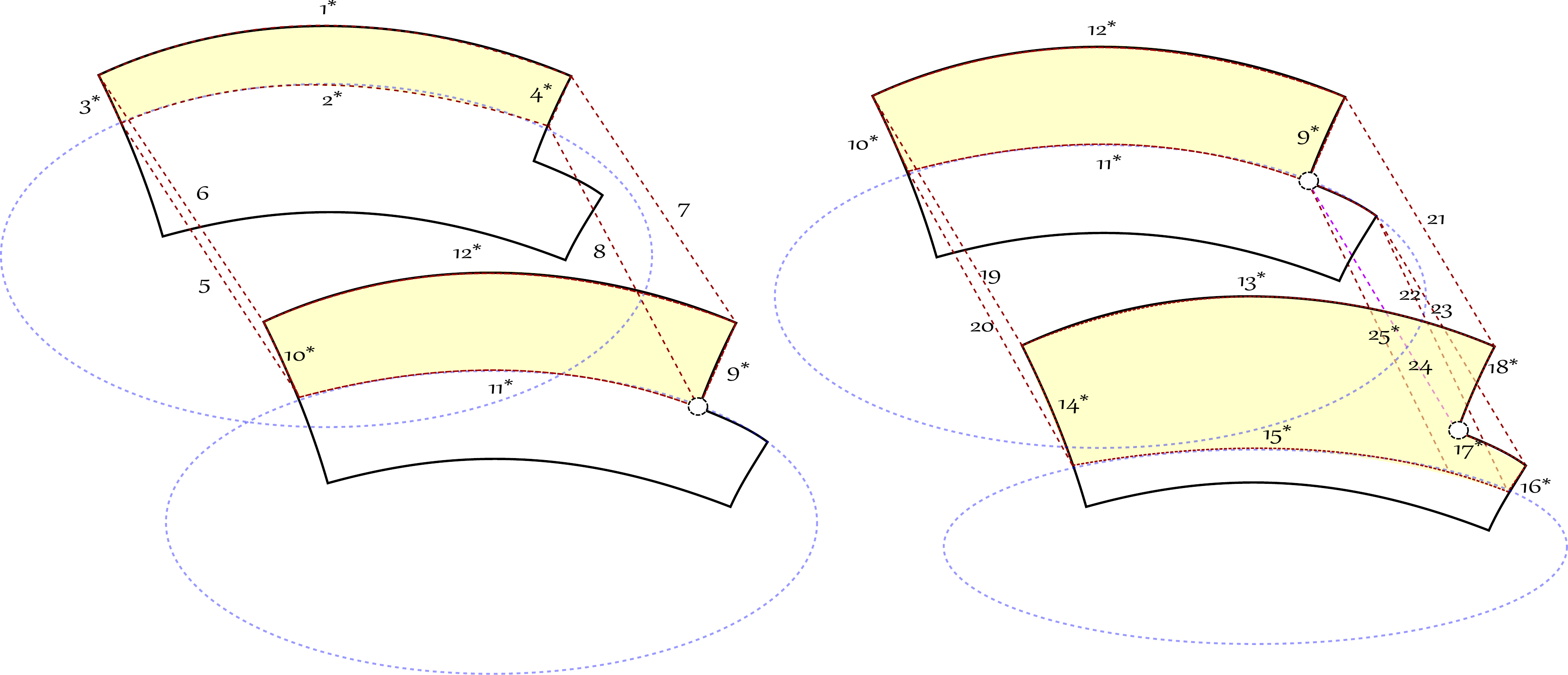}
				\caption{Пример 1-клеток комплекса $U_{\lambda_2}$ для биллиарда, изображённого на рисунке. Желтым обозначены области возможного движения. Слева изображены 1-клетки, соответствующие значениям дополнительного интеграла $\varLambda \leq c$, справа --- $\varLambda \geq c$. Всего в комплексе $U_{\lambda_2}$ будет 46 одномерных клетки (они отмечены красным цветом), из них 14 клеток первого типа и 32 второго (множество $Q$ состоит из $15$ элементов). Если номер отрезка помечен звездочкой, то в данный отрезок проектируется две клетки из $U_{\lambda_2}$ при естественной проекции $\pi$, если двумя звездочками --- четыре, иначе --- одна. Белым отмечены особые точки.} 
				\label{1cell} 
			\end{minipage}
		\end{center}
	\end{figure}
	\item Построим 2-клетки, для этого вновь рассмотрим множество граничных квадрик $Q$, построенных в предыдущем пункте. 2-клеток вновь будет два типа, то есть $U^{2}_{(\lambda_i)} = U^{2,1}_{(\lambda_i)} \cup U^{2,2}_{(\lambda_i)}$. Первый тип 2-клеток --- это связные компоненты при естественной проекции $\pi$ квадрик из $Q$ при всех значениях дополнительного интеграла $\varLambda$ из множества $\varLambda \in (c-\varepsilon, c) \cap (c, c + \varepsilon)$. 
	Второй тип 2-клеток --- это связные прообразы пересечений областей возможного движения $\varOmega^{-}_{c}$, $\varOmega_{c}$ и $\varOmega^{+}_{c}$ с элементарными биллиардами разбиения $\varSigma_1, \ldots, \varSigma_N$ при фиксированных значениях дополнительного интеграла $\varLambda = c - \varepsilon$, $\varLambda = c$ и $\varLambda = c + \varepsilon$ соответственно (см. пример на рис. \ref{2cell}). 
	
		\begin{figure}[h!]
		\begin{center}
			\begin{minipage}[h]{1.0\linewidth}
				\includegraphics[width=1\linewidth]{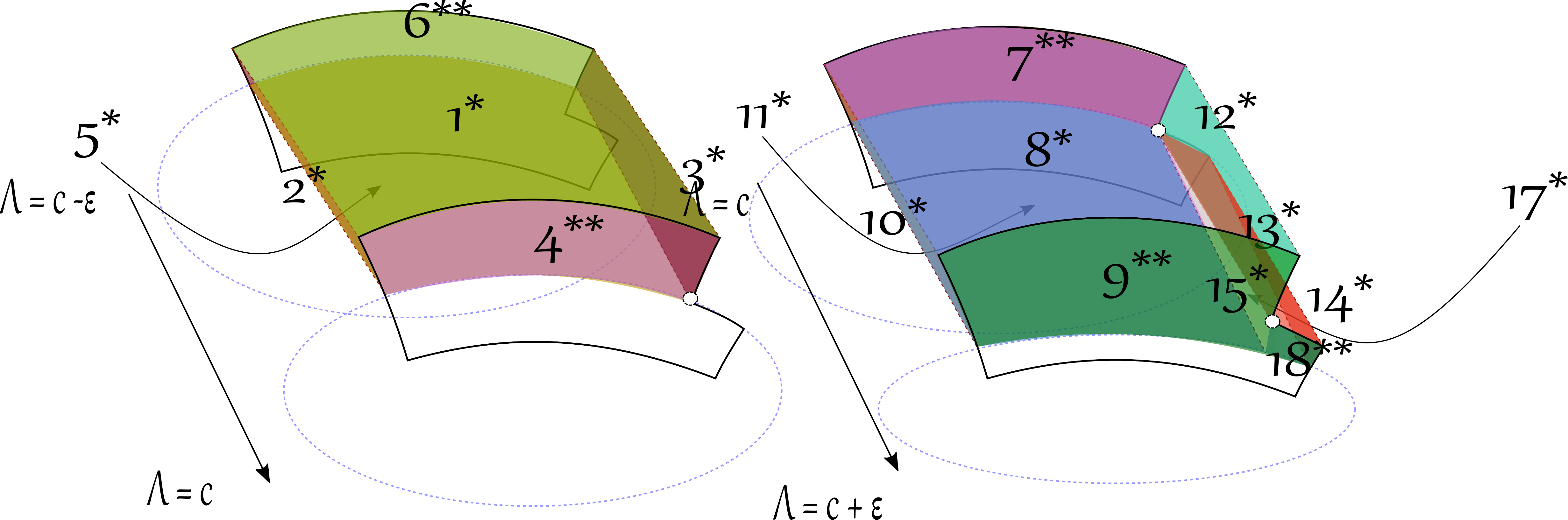}
				\caption{Пример 2-клеток комплекса $U_{\lambda_2}$ для биллиарда, изображённого на рисунке. Слева изображены 2-клетки, соответствующие значениям дополнительного интеграла $\varLambda \leq c$, справа --- $\varLambda \geq c$. Всего в комплексе $U_{\lambda_2}$ будет 46 двумерные клетки. Из них 26 клетки первого типа и 20 клеток второго типа. Если номер поверхности помечен звездочкой, то в данную поверхность проектируется две клетки из $U_{\lambda_2}$ при естественной проекции $\pi$, если двумя звездочками --- то четыре, иначе --- одна. Белым отмечены особые точки.} 
				\label{2cell} 
			\end{minipage}
		\end{center}
	\end{figure}
	\item Завершим построение комплексов $U_{(\lambda_i)}$, построив трехмерные клетки. Рассмотрим произвольное значение интеграла $\varLambda = \alpha \in [c-\varepsilon, c+ \varepsilon]$ и область возможного движения $\varOmega_{\alpha}$ для такого значения интеграла. Рассмотрим связные компоненты прообразов при естественной проекции $\pi$ пересечений внутренностей областей возможного движения $\varOmega_{\alpha}$ с элементарными биллиардами разбиения $\varSigma_1, \ldots, \varSigma_{N}$. Эти связные прообразы еще раз разобьем на две части по значениям дополнительного интеграла $\varLambda$, разрезав клетки по линии уровня функции $\varLambda = c$ (см. пример на рис. \ref{3cell}). 
		\begin{figure}[h!]
		\begin{center}
			\begin{minipage}[h]{1.0\linewidth}
				\includegraphics[width=1\linewidth]{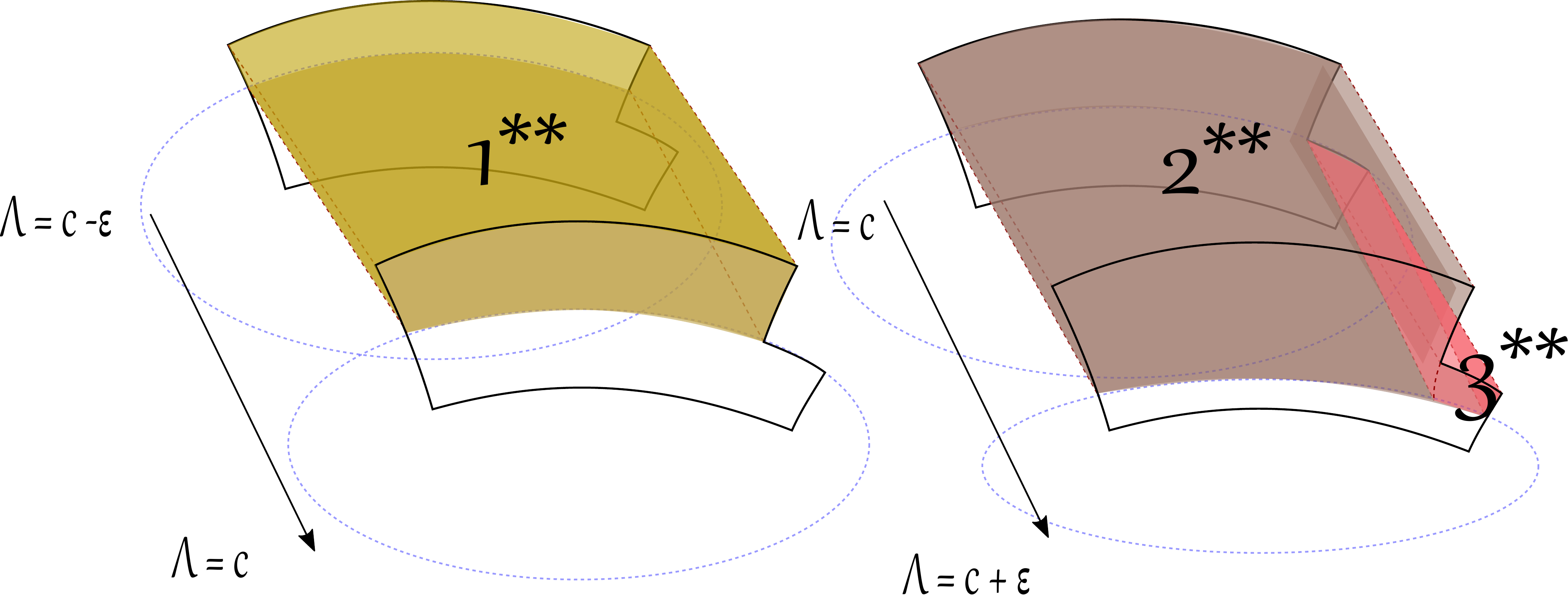}
				\caption{Пример 3-клеток комплекса $U_{\lambda_2}$ для биллиарда, изображённого на рисунке. Слева изображены 3-клетки, соответствующие значениям дополнительного интеграла $\varLambda \leq c$, справа --- $\varLambda \geq c$. Всего в комплексе $U_{\lambda_2}$ будет 12 трехмерных клеток. В каждую трехмерную поверхность проектируется четыре клетки из  $U_{\lambda_2}$.} 
				\label{3cell} 
			\end{minipage}
		\end{center}
	\end{figure}
\end{enumerate}
\begin{statement}
	Рассмотрим однородный биллиард $\varOmega$ с выбранным на нем разбиением $\varSigma_1, \ldots, \varSigma_{N}$ c граничными гиперболами (эллипсами) $\lambda_1, \ldots, \lambda_n$. Данное выше определение ком\-плек\-сов $U$ и $U_{\lambda_i}$ корректно, то есть, все указанные клетки действительно гомео\-морфны дискам соответствующей размерности и все свойства $CW$-комплексов выпол\-нены.
\end{statement}
\begin{proof}
	\begin{enumerate}
		\item Проверим, что $n$-мерные клетки действительно гоме\-омор\-фны $n$-мерным дискам и что граница $i$-мерной клетки содержится в объединении клеток меньшей размерности. Без ограничения общности, рассмотрим однородно-гиперболический биллиард $\varOmega$ с выбранным на нем разбиением на элементарные биллиарды $\varSigma_1, \ldots, \varSigma_N$. Рассмотрим сначала не седловой особый уровень дополнительного интеграла $\varLambda = \lambda_i$. В следствие теоремы Якоби-Шаля [1], в любой внутренней точке $x$ биллиарда $\varOmega$ можно выбрать четыре различных вектора скорости, направленных по касательной к каустике. В вершинах углов $\pi/2$ все четыре вектора скорости склеиваются в один в силу определения отражения в углах (см. опр. 1.1), а в вершинах углов $3\pi/2$ векторы скорости не определены. Вне вершин углов, на тех сегментах границы биллиарда $\varOmega$, которые попадают в область возможного движения, векторы скорости попарно склеиваются в силу биллиардного закона отражения, т.е. в таких точках биллиарда определены только два несовпадающих вектора скорости. Также два вектора скорости определены в точках, лежащих на каустике, так как в таких точках можно задать только два лежащих касательных вектора, которые лежат на одной прямой.
		
		В случае $\varLambda = b$ в силу оптического свойства квадрик, на таком уровне вектора скорости всех точек обязаны лежать на прямых, проходящих через фокусы (каустика на данном уровне не определена). В таком случае, два вектора скорости будут определены также в точках фокальной прямой и только один вектор скорости будет определён в точках пересечения фокальной прямой и границы биллиарда. 
		
		Очевидно, что 0-клетки комплекса $U_{(\lambda_i)}$ гомеоморфны 0-мерным дискам. Далее, рассмотрим одномерные клетки комплексов $U_{(\lambda_i)}$ первого типа.
		В точке $x$ из $P$ может быть определен либо один единственный вектор скорости, либо два. Тогда при изменении дополнительного интеграла $\varLambda$ вектор скорости в точке $x$ будет меняться непрерывно, а потому множества $(x,v)$, где $ x \in P$ будут гомеоморфны либо одному отрезку (если в точке $x$ был определен $1$ вектор), либо двум (если в точке $x$ было определено $2$ вектора). Также заметим, что границами 1-клеток первого будут 0-клетки, так как это будут точки вида $(x,v), x \in P, \varLambda = c \pm \varepsilon, \varLambda = c$, что в точности совпадает с 0-клетками.
		Второй тип 1-клеток гомеоморфен отрезкам и его границами являются 0-клетки по определению. 
		
		Рассмотрим 2-клетки. Начнем с 2-клеток первого типа и рассмотрим сегмент границы биллиарда $\varOmega$ из $Q$ (ограниченный точками из $P$). В каждой точке $x$ такого сегмента $q$, на каждом уровне интеграла $\varLambda = \alpha$, можно определить два вектора скорости, следовательно, каждая 2-клетка гомеоморфна произведению квадрики $q$ на отрезок (так как мы выбираем связную компоненту). В итоге, каждая 2-клетка первого типа гомеоморфна 2-диску и ее границей будут 1-клетки первого типа (пары $(x,v)$, c фиксированной точкой $x$). 
		Легко заметить, что границами таких клеток будут 1-клетки. 2-клетки второго типа, очевидно, гомеоморфны 2-дискам и их границами являются 1-клетки второго типа.
		
		Теперь рассмотрим 3-клетки. В каждой внутренней точке $x$ области возможного движения $\varOmega_{\alpha}$ определены четыре вектора скорости и $\varOmega_{\alpha} \cong \varOmega_{\beta}$, если и только если $\alpha$ и $\beta$ одновременно больше $c$, одновременно равны $c$, либо одновременно меньше $c$. Следовательно каждая двумерная клетка гомеоморфна $A \times I$, где $A$ --- область биллиарда, гомеоморфная пересечению $\varOmega_{\alpha}$ с некоторым элементарным биллиардом разбиения $\varSigma_j$.
		Также легко заметить, что их границами являются 2-клетки первого и второго типа.
		
		Аксиомы $C$ и $W$ выполнены в силу конечности размерности и количества клеток каждой размерности.  
	\end{enumerate}
\end{proof}
\section{Окрестности особых слоев $\Lambda = \lambda_i$. Описание топологии атомов $U_{\lambda_i}$.}
Особый интерес при исследовании биллиардов с невы\-пук\-лыми углами на границе бил\-лиардной области пред\-ставляют атомы на уровне $\varLambda = \lambda_i$, где $\lambda_i$ --- параметр квадрики на которой лежат вершины углов $3\pi/2$. При переходе через эти критические уровни меняется род регулярной поверхности, так как в области возможного движения меняется количество особых точек. В выпуклых биллиардах не существует аналогов таких слоев. 

Перейдем к рассмотрению случаю $\varLambda = \lambda_i$ и, соответственно, к изучению топологии трехмерного комплекса $U_{\lambda_i}$. Зафиксируем плоский, односвязный и однородный биллиард на плоскости $\varOmega$. Также без ограничения общности будем предполагать, что биллиард однородно-гиперболический, что особые точки лежат только на правой дуге гиперболы и что вблизи гиперболы $\lambda_i$ нет других дуг граничных гипербол, содержащих особые точки (см. замечания после Теоремы 3). Начнем изучение топологии комплекса $U_{\lambda_i}$ с определения комплекса $\tilde U_{\lambda_i} \subset U_{\lambda_i}$, проектирующегося при естественной проекции $\pi: Q^3 \rightarrow \varOmega$ на область биллиарда $\varOmega_{\lambda_i} = \{x \in \varOmega~|~ d(x, \lambda_i) \leq \varepsilon, \exists (x,v) \in U_{\lambda_i} : \pi(x,v) = x \}$, где $d(x,y)$ --- стандартная метрика на плоскости, то есть в участок области возможного движения при фиксированном интеграле $\varLambda = \lambda_i$ вблизи гиперболы с параметром $\lambda_i$.  Введем некоторые обозначения. Рассмотрим область биллиарда $\varOmega_{\lambda_i}$, разрежем ее по гиперболе $\lambda_i$. Обозначим через $\nu$ число компонент связности внутри гиперболы $\lambda_i$, а через $\xi$ --- число компонент связности вне гиперболы $\lambda_i$ (см. рис. \ref{GammaObl}).
	\begin{figure}[!ht]
	\centering
	\includegraphics[width=0.9\textwidth]{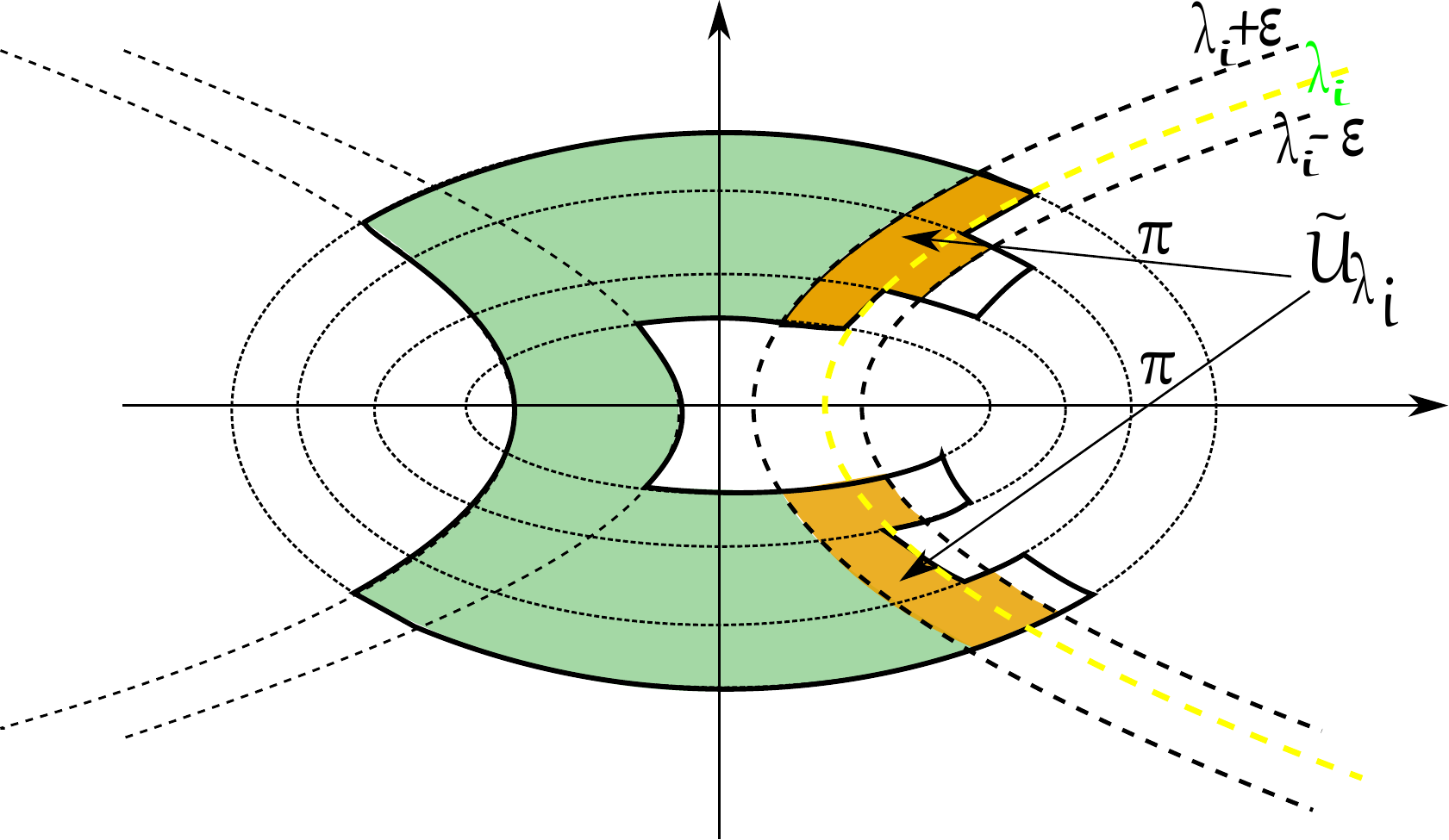}
	\caption{Пример. Область $\varOmega_{\lambda_i}$ --- это область между двумя гиперболами с параметрами $\lambda_i+\varepsilon$ и $\lambda_i - \varepsilon$. Здесь $\nu =2, \xi =3$.}
	\label{GammaObl}
\end{figure}

Перейдем от разбиения $\varSigma_1, \ldots, \varSigma_N$ биллиарда $\varOmega$ к разбиению <<полосы>> $\varOmega_{\lambda_i}$. Рассмотрим элементарные биллиарды $\varSigma'_j = \varSigma_j \cap \varOmega_{\lambda_i}$. Получившиеся в итоге разбиение участка биллиарда $\varOmega_{\lambda_i}$ обозначим через $\varSigma'_1, \ldots, \varSigma'_t, \ldots, \varSigma'_{\xi+\nu}$, где $\xi + \nu \leq N$.  
\subsection{Определение разрезающих комплексов $T_i$.}
Рассмотрим трехмерный комплекс $\tilde U_{\lambda_i}$, определенный выше. Определим двумерный подкомплекс $T_i$ в трехмерном комплексе $\tilde U_{\lambda_i}$, который будет прообразом граничной дуги при естественной проекции $\pi$ при фиксированном дополнительном интеграле $\varLambda \in [\lambda_i - \varepsilon, \lambda_i + \varepsilon]$.
\begin{definition}
	Рассмотрим однородно-гиперболический (эллиптический) биллиард $\varOmega$ с выбранным на нем разбиением $\varSigma_1, \ldots, \varSigma_N$ и рассмотрим двумерный остов $T_i \subset \tilde U_{\lambda_i}$ трехмерного подкомплекса $\tilde U_{\lambda_i} \in U_{\lambda_i}$, где клетки выбраны как показано в главе 5. 
\end{definition}
\subsection{Описание комплекса $U_{\lambda_i} \setminus \tilde U_{\lambda_i}$.}
Перед подробным анализом топологии трехмерного комплекса $\tilde U_{\lambda_i}$ сначала рассмотрим трехмерный комплекс $U_{\lambda_i} \setminus \tilde U_{\lambda_i}$, где пары $(x,v) \in U_{\lambda_i} \setminus \tilde U_{\lambda_i}$  проектируются естественной проекцией $\pi$ в биллиардную область <<вне дуги $\lambda_i$>>, то есть, в область биллиарда $\varOmega \setminus \varOmega_{\lambda_i}$. Заметим, что движение в области $\varOmega \setminus \varOmega_{\lambda_i}$ на уровнях дополнительного интеграла $\varLambda \in [\lambda_i-\varepsilon, \lambda_i+ \varepsilon]$ происходит аналогично движению на неособом уровне, так как область возможного движения меняется в окрестности гиперболы $\lambda_i$, а каждая точка внутренняя области $\varOmega \setminus \varOmega_{\lambda_i}$ по прежнему оснащается четырьмя векторами скорости. 

Для этого рассмотрим произвольный однородно-гиперболический биллиард $\varOmega$ и его произвольный неособый слой $\varLambda = \alpha$, описывающийся теоремой 2. Теперь обрежем этот биллиард по произвольной гиперболе, обозначим получившийся <<обрезанный>> биллиард через $\hat \varOmega$ и найдем чему теперь гомеоморфен слой $\varLambda = \alpha$.   
\begin{statement}\label{CutBilliard}
Рассмотрим однородно-гиперболический биллиард $\varOmega$ и неособый слой дополнительного интеграла $\varLambda  = \alpha$. Теперь разрежем биллиард $\varOmega$ по дуге гиперболы с параметром $\theta$, не содержащей вершин углов, обозначим получившийся биллиард (внутри дуг разрезанной гиперболы) через $\hat \varOmega$.  Тогда полный прообраз при проекции $\pi$ биллиарда $\hat \varOmega$ в $Q^3$ гомеоморфен поверхности рода $g+1$ с $g$ выколотыми и где $\nu$ ручек этой поверхности разрезаны. Здесь $g$ --- количество особых точек, лежащих внутри области возможного движения при дополнительном интеграле $\varLambda = \alpha$ и лежащих внутри биллиарда $\hat \varOmega$. А $\nu$ --- число связных компонент граничной дуги гиперболы с параметром $\theta$. 
\end{statement}
\begin{proof}
	В силу теоремы 2 [4] неособый слой дополнительного интеграла $\varLambda$ у биллиарда $\hat \varOmega$ был бы гомеоморфен поверхности рода $g+1$ с $g$ выколотыми точками, если бы на гиперболе разреза $\theta$ выполнялся бы биллиардный закон.   Обозначим этот двумерный (значение обоих интегралов фиксировано) комплекс через $G$. Выберем заполнение биллиарда $\varOmega$ софокусными гиперболами семейства, тогда мы также сможем расслоить $G$ на одномерные прообразы гипербол расслоения при естественной проекции $\pi$. Найдем прообраз произвольной гиперболы из расслоения, оснащенной векторами скорости, направленными вправо. Каждая внутренняя точка $x$ биллиарда $\varOmega$, лежащая в области возможного движения, может быть оснащена двумя векторами скорости, каждая точка $x$, лежащая на границе биллиарда ---  двумя, которые склеиваются по биллиардному закон. Тогда прообраз такой оснащенной гиперболы --- $\nu$ окружностей, где $\nu$ ее число компонент связности. Тогда отмена биллиардного закона на гиперболе $\theta$ --- это разрез $\nu$ ручек.  
\end{proof}
\subsection{Описание не седловых бифуркационных слоев.}
Следующая теорема описывает топологию трёхмерного комплекса $\tilde U_{\lambda_i}$ (см. рис. \ref{Table1}).
\begin{theorem}\label{ThSing1}
	Рассмотрим однородно-гиперболический биллиард $\varOmega$ с выбранным на нем разбиением $\varSigma_1, \ldots, \varSigma_N$. Рассмотрим окрестность $\varLambda \in [\lambda_i-\varepsilon, \lambda_i+ \varepsilon]$ особого слоя второго интеграла $\varLambda = \lambda_i$ и, соответственно, трехмерные комплексы $U_{\lambda_i}$ и $\tilde U_{\lambda_i}$. Также рассмотрим область $\varOmega_{\lambda_i} \subset \varOmega$ и перейдем к ее фокальному разбиению на элементарные биллиарды $\varSigma'_1, \ldots, \varSigma'_{\nu+\xi}$. Вырежем из комплекса $U_{\lambda_i}$ трехмерный комплекс $\tilde U_{\lambda_i}$, а из  $\tilde U_{\lambda_i}$ вырежем двумерный разрезающий комплекс $T_i$. Тогда:
	\begin{enumerate}
		\item Трехмерный комплекс $U_{\lambda_i} \setminus \tilde U_{\lambda_i} \cong G_g \times I$, где $G_g$ --- поверхность рода $g+1$ с $g$ выколотыми и где $\nu$ ручек этой поверхности разрезаны. Здесь $g$ --- количество особых точек, лежащих внутри области возможного движения при дополнительном интеграле $\varLambda = \alpha \in [\lambda_i-\varepsilon, \lambda_i+ \varepsilon]$ и лежащих внутри биллиарда $\varOmega \setminus \varOmega_{\lambda_i}$; 
		\item Двумерный комплекс $T_i \cong Gr_i \times I$, где $Gr_i$ --- граф, построенный по Алгоритму 1;
		\item Трехмерный комплекс $\tilde U_{\lambda_i} \setminus T_i |_{\varLambda = \alpha, \alpha \in [\lambda_i - \varepsilon, \lambda_i]} \cong (C_1 \cup \ldots \cup C_{2\nu+2\xi}) \times I$; \\
		Трехмерный комплекс $\tilde U_{\lambda_i} |_{\varLambda = \alpha, \alpha \in [\lambda_i, \lambda_i+\varepsilon]} \cong (C_1 \cup \ldots \cup C_{2\nu}) \times I$;
		\item Трехмерный комплекс $\tilde U_{\lambda_i} \setminus T_i |_{\varLambda = \alpha, \alpha \in [\lambda_i - \varepsilon, \lambda_i]}$ приклеивается к двумерному комплексу $T_i$ послойно: на каждом уровне $\varLambda = \alpha, \alpha < \lambda_i$ все цилиндры $(C_1 \cup \ldots \cup C_{2\nu+2\xi})$ приклеиваются к графу $Gr_i$ по Алгоритму 2, а при $\varLambda = \lambda_i$ к графу $Gr_i$ приклеиваются только цилиндры $C_1 \cup \ldots \cup C_{2\nu}$ (Алгоритм 2 применяется только для элементарных биллиардов $\varSigma'_1, \ldots, \varSigma'_N$).
	\end{enumerate}
\end{theorem}
\begin{figure}[h!]
	\begin{center}
		\begin{minipage}[h]{1.0\linewidth}
			\includegraphics[width=0.9\linewidth]{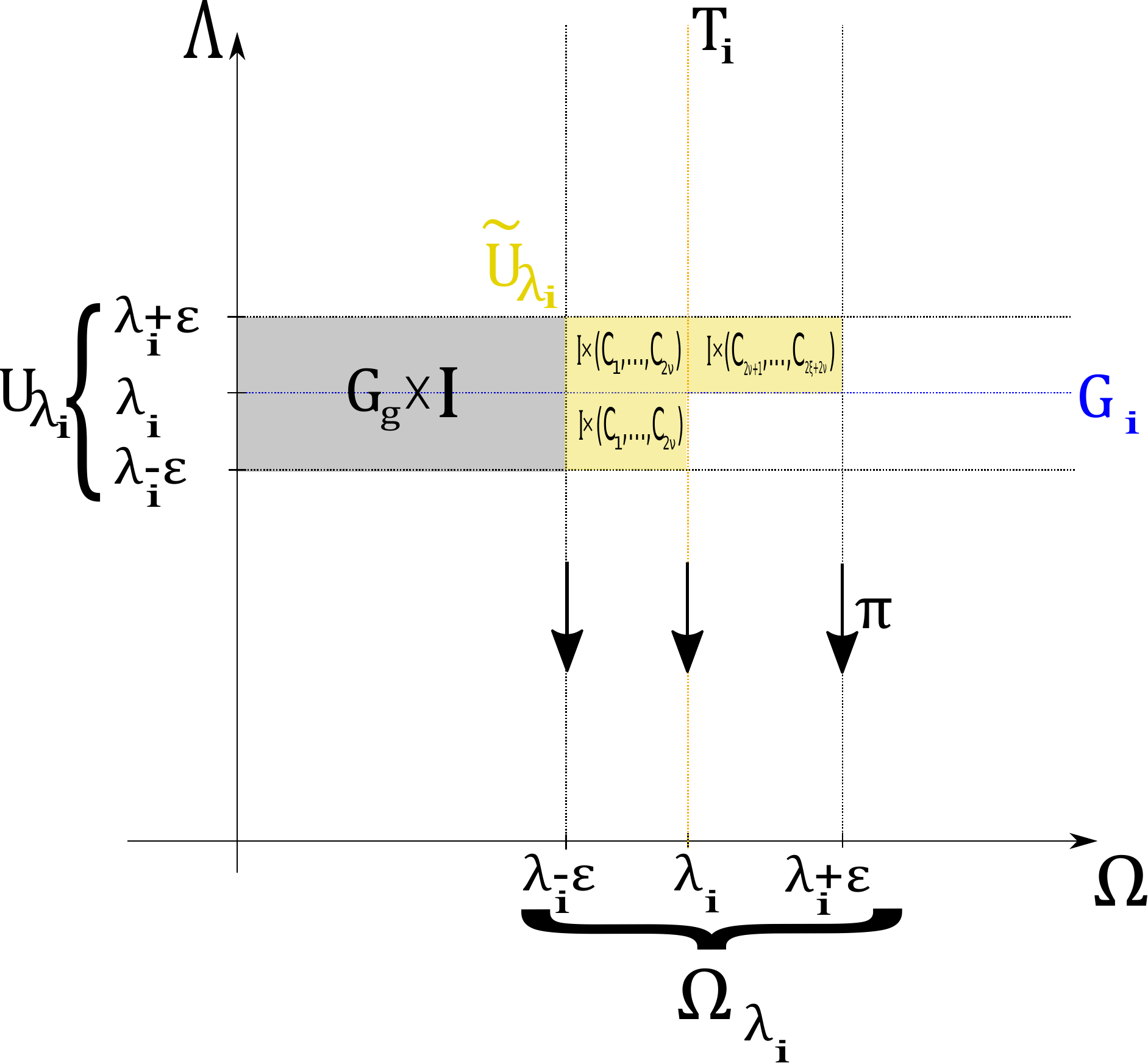}
			\caption{Иллюстрация теоремы 3. По оси абсцисс обозначены значения естественной проекции $\pi$, а по оси ординат --- значения дополнительного интеграла $\varLambda$.} 
			\label{Table1} 
		\end{minipage}
	\end{center}
\end{figure}
\textbf{Алгоритм 1. Алгоритм построения графа $Gr_i$ --- прообраза дуги $\lambda_i$ при естественной проекции $\pi$ при любом фиксированном значении дополнительного интег\-рала $\varLambda = \alpha \neq b$.}
\begin{figure}[h!]
	\begin{center}
		\begin{minipage}[h]{1.0\linewidth}
			\includegraphics[width=0.9\linewidth]{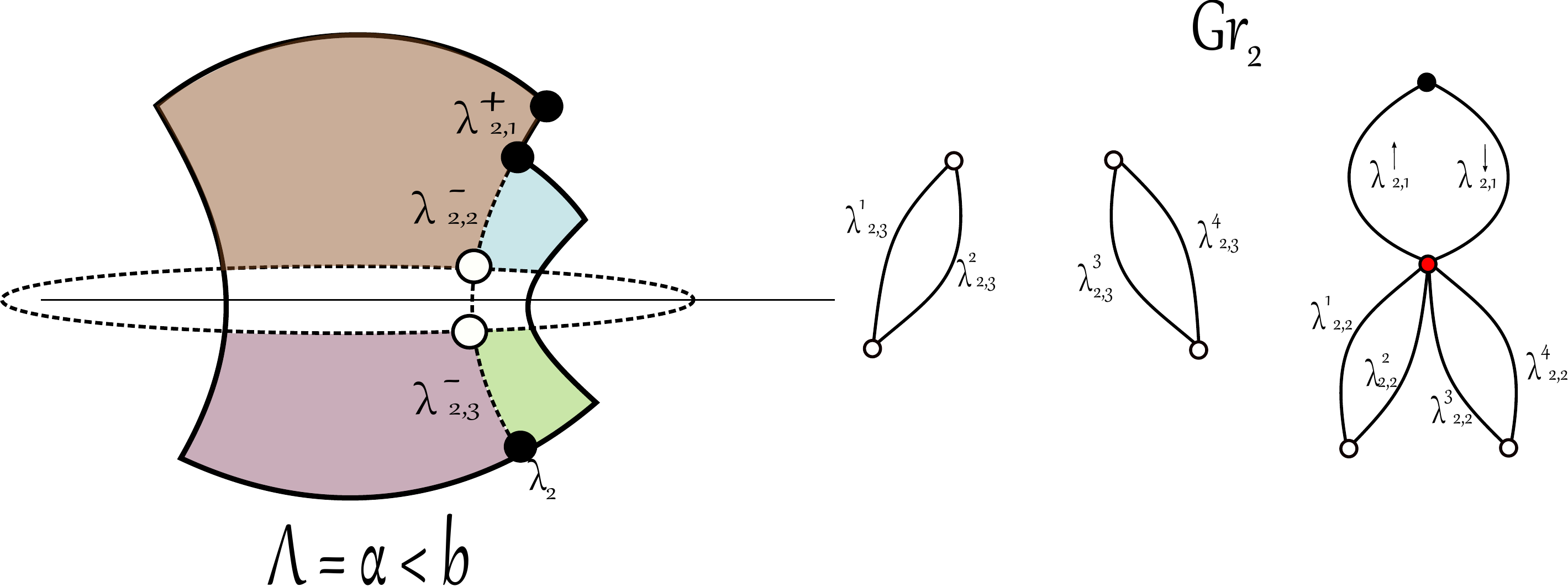}
			\caption{Пример построения графа $Gr_2$ по Алгоритму 1 для дуги граничной гиперболы $\lambda_2$ биллиарда, изображенного на рисунке.} 
			\label{Alg1} 
		\end{minipage}
	\end{center}
\end{figure}
\begin{enumerate}
	\item[Шаг 1.] Рассмотрим ту часть граничной дуги $\lambda_i$ биллиарда $\varOmega$, которая попадает в область возможного движения $\varOmega_{\alpha}$ при заданном значении интеграла $\varLambda = \alpha \neq b$. Рассмотрим пересечения этой дуги $\lambda_i$ с границами области возможного движения биллиарда $\varOmega_{\alpha}$. На этой части дуги $\lambda_i$ отметим черными точками $$ b^{i}_1, \ldots,  b^{i}_{n^b_i}$$ все вершины углов $\pi/2$ или $3\pi/2$. Белыми точками $$ w^i_1, \ldots,  w^i_{n^w_i}$$ отметим пересечения дуги $\lambda_i$ с границами области возможного движения $\varOmega_{\alpha}$ вне вершин углов. Обозначим множество черных точек на граничной дуге $\lambda_i$ через $B_i$, а белых --- $W_i$. Теперь рассмотрим сегменты граничной дуги $\lambda_i$, ограниченные отмеченными точками и попадающие в область возможного движения. Назовем их $\lambda_{i,t}$, введя на дуге ${\lambda_i}$ естественную нумерацию сегментов $t$, здесь $t \in 1, \ldots, n^w_i+n^b_i$. Дополнительно пометим сегменты граничных дуг $ \lambda_{i,t}$ символом <<$+$>>, если сегмент является частью границы области возможного движения $\varOmega_{\alpha}$, и символом <<$-$>>, если сегмент $ \lambda_{i,t}$ является границей между элементарными биллиардами разбиения, но не является сегментом границы объемлющей области возможного движения $\varOmega_{\alpha}$.

	\item[Шаг 2.] Рассмотрим граничную дугу $\lambda_i$ и завершим построение $Gr_i \cong \pi^{-1}(\lambda_i) |_{\varLambda = \alpha}$. Каждой черной точке из множества $B_i$ соответствует одна вершина графа $$b^{i}_1, \ldots, b^{i}_{n^b_i}.$$ Каждой белой точке из множества $W_i$ соответствуют две вершины (всего $2n^w_i$ вершин) графа $Gr_i$, обозначаемые $$w^{i,r}_1, w^{i,l}_1, \ldots, w^{i,r}_{n^w_i}, w^{i,l}_{n^w_i}.$$ Каждому граничному сегменту $\lambda^{+}_{i,t}$ будут соответствовать два ребра (векторы в $Q^3$ будут направлены вверх и вниз) $$\lambda^{\uparrow}_{i,t}, \lambda^{\downarrow}_{i,t} \in Gr_{\lambda_i},$$ а каждому граничному сегменту $ \lambda^{-}_{i,t}$ --- четыре: $$\lambda^{1}_{i,t}, \lambda^{2}_{i,t}, \lambda^{3}_{i,t}, \lambda^{4}_{i,t} \in Gr_{\lambda_i}.$$ Теперь расставим на ребрах $\lambda^{v}_{i,t}$ графа $Gr_{\lambda_i}$ буквы
	$$b^{i}_1, \ldots, b^{i}_{n^b_i}, w^{i,r}_{n^b_i+1}, w^{i,l}_{n^b_i+1}, \ldots, w^{i,r}_{n^w_i+n^b_i}, w^{i,l}_{n^w_i+n^b_i}$$ по таблицам на рис. \ref{Graphes}. Теперь повторим Шаг 2 для всех значений $i \in 1, \ldots, n$ и для всех значений $t \in 1, \ldots, n^w_i+n^b_i$. Проведем склейку по совпадающим буквам. Выкалываем все черные точки графов $Gr_{\lambda_i}$, соответствующие особым точками биллиарда $\varOmega$.
\end{enumerate}

Теперь нужно указать способ, которым цилиндры $C_1 \cup \ldots \cup C_{t+s}$ приклеиваются к графам $Gr_i$. 

\textbf{Алгоритм 2. Склейка графов $Gr_i$ и цилиндров (двумерных прообразов при естественной проекции $\pi$ элемен\-тарных биллиардов $\varSigma_j$ без границ) $C_1, \ldots, C_{\nu+\xi}$ на произвольном не седловом уровне допол\-нительного интеграла $\varLambda = \alpha \neq b$.}
\begin{enumerate}
	\item [Шаг 1.] Рассмотрим два двумерных цилиндра $S^1 \times I$ для каждого элемента фокального разбиения $\varSigma'_j$ области возможного движения $\varOmega_{\alpha}$ (см. главу 4.2). Обозначим один из них через ${C_j}^L$, а другой --- $C_j^R$. Будем помечать граничные окружности каждого цилиндра буквами $\lambda^{v}_{i,t}$ ($v=1,2,3,4,\uparrow, \downarrow$) по следующему правилу: 
	Рассмотрим элементарный биллиард $\varSigma'_{j}$ и одну из его дуг граничных гипербол с параметром $\lambda_i$. Пусть на этой границе биллиарда $\varSigma'_{j}$ отмечены $L+L'$ сегментов граничных гипербол $$\lambda^{+,-}_{i,1}, \ldots, \lambda^{+,-}_{i,L+L'},$$ где на сегменте выставлен либо символ <<$+$>> (таких сегментов $L$), либо символ <<$-$>> (таких сегментов $L'$). Рассмотрим одну из граничных окружностей $S_R^1$ цилиндра $C^R_{j}$ (аналогично окружность $S_L^1$ цилиндра $C^L_{j}$). Разделим окружность $S_R^1$ на $2L+2L'$ частей. Выберем точку на окружности $S_R^1$ на границе любых из $2L+2L'$ частей. Будем ставить в соответствие частям окружности $S^1_R$ ($S^1_L$) сегменты $\lambda^{+,-}_{i,t}$. Поставим в соответствие первому $\lambda^{+,-}_{i,1}$ две части окружности: справа и слева от отмеченной на $S^1_R$($S^1_L$) точки. Следующему сегменту $\lambda^{+,-}_{i,t}$ будут соответствовать две следующие части окружности $S^1_R$ ($S^1_L$): опять же справа и слева от точки, и.т.д. В результате каждому сегменту граничной гиперболы $ \lambda^{+,-}_{i,t}$ поставим в соответствие четыре сегмента граничных окружностей (отрезка) $S_R^1$ и $S_L^1$. Соответствующие сегменту $ \lambda^{+}_{i,t}$ отрезки помечаются следующим образом: на обоих цилиндрах ставятся одинаковые буквы $\lambda^{\uparrow}_{i,t}, \lambda^{\downarrow}_{i,t}$. Соответствующие сегменту $ \lambda^{-}_{i,t}$ отрезки помечаются следующим образом: на $S_R^1$ ставим $\lambda^{1}_{i,t}, \lambda^{2}_{i,t}$, а на $S_L^1$ ставим $\lambda^{3}_{i,t}, \lambda^{4}_{i,t}$.	
	\item[Шаг 2.]  Повторим Шаг 1 для всех элементарных биллиардов $\varSigma'_j$ из фокального разбиения области возможного движения $\varOmega_{\alpha}$ на элементарные биллиарды. 
\end{enumerate}

\begin{proof}
	Рассмотрим однородно-гиперболический биллиард $\varOmega$, его подобласть $\varOmega_{\lambda_i} \subset \varOmega$ и разбиение области $\varOmega_{\lambda_i}$ на элементарные биллиарды $\varSigma'_1, \ldots, \varSigma'_{t+s}$. Проведем дугу гиперболы с параметром $\lambda_i$. 
	
	Зафиксируем произвольное значение второго интеграла $\varLambda = \alpha, \alpha \in [\lambda_i - \varepsilon, \lambda_i + \varepsilon]$. Рассмотрим область возможного движения $\varOmega_{\alpha}$, находящуюся внутри <<полосе>> $\varOmega_{\lambda_i}$, а именно, рассмотрим область $\tilde \varOmega_{\alpha} = \varOmega_{\lambda_i} \cap \varOmega_{\alpha}$ и ее разбиение на элементарные биллиарды (см. 4.2). 
	В следствие теоремы Якоби-Шаля [1], в любой внутренней точке $x$ биллиарда $\varOmega$ можно выбрать четыре различных вектора $v_1, v_2, v_3, v_4$ скорости, направленных по касательной к каустике. 
	Расслоим область $\tilde \varOmega_{\alpha}$ на невырожденные гиперболы из софокусного семейства. Рассмотрим множества в $\tilde U_{\lambda_i}$, задаваемые как пары $(x \in \varSigma'_j, v_1)$, где $v_1$ --- вектор, направленный от левого фокуса и пары $(x \in \varSigma'_j, v_2)$, где $v_2$ --- вектор, направленный к правому фокусу. На эллиптических границах биллиардной области $ \varOmega_{\alpha}$ множества пар $(x, v_1)$ и $(x, v_2)$ склеятся по биллиардному закону, а, следовательно, множества пар $(x \in \varSigma'_j, v_1)$ и $(x \in \varSigma'_j, v_2)$ склеются в $\tilde U_{\lambda_i}$ по прообразу этой границы при естественной проекции $\pi$. Пары $(x \in \varSigma'_j, v_3)$ и $(x \in \varSigma'_j, v_4)$ склеются аналогично на гиперболических границах биллиарда. Теперь осталось заметить, что $(x \in \varSigma'_j, x \notin \partial \varSigma'_j, v_l), l = 1,2,3,4$ гомеоморфно двумерному диску $I \times I$, так как все элементарные биллиарды разбиения имеют форму квадрата с границами из эллипсов и гипербол. В результате, заметим, что каждому элементарному биллиарду разбиения (без границы) соответствует два двумерных цилиндра $C^R_j$ и $C^L_j$.  
	\begin{enumerate}
		\item Трехмерный комплекс  $U_{\lambda_i} \setminus \tilde U_{\lambda_i}$ послоен на двумерные комплексы, соот\-вет\-ствующие значениям второго интеграла $\varLambda = \alpha, \alpha \in [\lambda_i - \varepsilon, \lambda_i + \varepsilon]$. В силу выбора достаточно малого $\varepsilon$ в области возможного движения при таких значениях интеграла будет лежать одинаковое количество вершин углов $3\pi/2$, а следовательно, в силу теоремы 2 на каждом уровне $\varLambda = \alpha$ в $U_{\lambda_i}/\tilde{U_{\lambda_i}}$ будут лежать поверхности одного рода (до разрезания). Дальше остается только применить Утверждение \ref{CutBilliard}.
		\item Рассмотрим дополнительный интеграл $\varLambda$. Если $\varLambda = \alpha, \alpha \in [\lambda_i-\varepsilon, \lambda_i]$, то дуга гиперболы $\lambda_i$ попадает в область возможного движения. Рассмотрим уровень $\varLambda = \alpha$, на дуге гиперболы $\lambda_i$ следующим образом определены вектора скорости: один вектор для вершин углов, два вектора для точек, лежащих на границе биллиарда $\varOmega$ и четыре вектора для точек, лежащих вне границы биллиарда. Прообразом дуги гиперболы $\lambda_i$ будет граф $\pi^-1(\lambda_i) |_{\varLambda = \alpha} \cong Gr_i$, причем этот граф не меняется при изменении $\alpha$ в пределах отрезка $[\lambda_i-\varepsilon, \lambda_i]$, так как количество прообразов при естественной проекции $\pi$ у точки зависит только от ее типа (граничная, угловая или особая) в биллиарде $\varOmega$. При $\varLambda = \alpha, \alpha \in (\lambda_i, \lambda_i+\varepsilon]$ гипербола $\lambda_i$ не попадает в область возможного движения $\varOmega_{\alpha}$, следовательно, прообраз будет пустым. 
		\item В силу рассуждения выше и его независимости от значения второго интеграла, если элементарный биллиард $\varSigma'_{j}$ попадает в область возможного движения при заданном значении интеграла $\varLambda$, то его прообразом при естественной проекции $\pi$ в $\tilde U_{\lambda_i}$ будут два цилиндра. После вырезания дуги гиперболы $\lambda_i$, область $\varOmega_{\lambda_i}$ распадется на элементарные биллиарды $\varSigma'_1, \ldots, \varSigma'_{\nu+\xi}$. Пока дополнительный интеграл $\varLambda \leq \lambda_i$, то в область возможного движения попадают все элементарные биллиарды $\varSigma'_1, \ldots, \varSigma'_{\nu+\xi}$, а после того, как интеграл $\varLambda$ превысит значение $\lambda_i$ --- только $\nu$ элементарных биллиардов, лежащих левее внутри гиперболы с параметром $\lambda_i$. 
		\item На каждом уровне $\varLambda = \alpha$ происходит склейка векторов $(x,v_u), x \in \lambda_i, u = 1,2,3,4$ по биллиардному закону. Это склейка поднимается естественной проекцией $\pi$ до послойной склейки комплекса $T_i$ и комплекса $\tilde U_{\lambda_i} \setminus T_i |_{\varLambda = \alpha, \alpha \in [\lambda_i - \varepsilon, \lambda_i]}$. Если $\varLambda = \lambda_i$, то в область возможного движения попадают только элементарные биллиарды фокального разбиения $\varSigma'_1, \ldots, \varSigma'_{\nu}$ и дуга гиперболы $\lambda_i$, так как только они лежат внутри гиперболы $\lambda_i$. Если $\varLambda < \lambda_i$, то в область возможного движения попадают все элементарные биллиарды $\varSigma'_1, \ldots, \varSigma'_{\nu+\xi}$. Конструкция склейки по биллиардному закону не зависит от уровня интеграла $\varLambda$, а потому, нужно выставить метки $\lambda^{+,-}_{i,t}$ на граничной дуге $\lambda_i$ всех элементарных биллиардов  $\varSigma'_1, \ldots, \varSigma'_{\nu+\xi}$, но при $\varLambda = \lambda_i$ последние $\xi$ элементарных биллиардов разбиения не будут приклеиваться к графу $Gr_i$. Зафиксируем произвольное значение интеграла $\varLambda = \alpha$. Рассмотрим множества $\tilde{\lambda^u_i} = ((x,v_u) \in Q^3| x \in \lambda_i, u = \rightarrow, \leftarrow)$ и разделим $\tilde{\lambda^u_i}$ на сегменты $\tilde{\lambda^u_{i,t}}$, аналогично Алгоритму 1. Поскольку на сегментах $\tilde{\lambda^u_{i,t}}$, которые помечены символом <<$+$>>, происходит склейка по биллиардному закону, то к этому множеству приклеиваются граничные окружности цилиндров $C^R_j$ и $C^L_j$ и склеиваются друг с другом. Если сегмент $\lambda_{i,t}$ помечен знаком <<$-$>>, то на нем склейки не происходит и, в результате, $\tilde{\lambda^u_{i,t}}$ не склеиваются друг с другом, но к ним приклеиваются граничные окружности цилиндров $C^R_j$, $C^R_{j'}$ и $C^L_j$, $C^L_{j'}$, где $\varSigma'_j$ и $\varSigma'_{j'}$ --- элементарные биллиарды с общим граничным сегментом $\lambda_{i,t}$. Двумерные цилиндры $C^R_j$, $C^R_{j'}$ приклеиваются по граничной окружности к  $\tilde{\lambda^u_{i,t}}$, где $u = (\rightarrow)$, а левые цилиндры соответственно к $\tilde{\lambda^u_{i,t}}$, где $u = (\leftarrow)$. Легко заметить, что Алгоритм 2 описывает идентичный механизм склейки. 
	\end{enumerate}
\end{proof}
\begin{figure}[h!]
	\begin{center}
		\begin{minipage}[h]{1.0\linewidth}
			\includegraphics[width=0.9\linewidth]{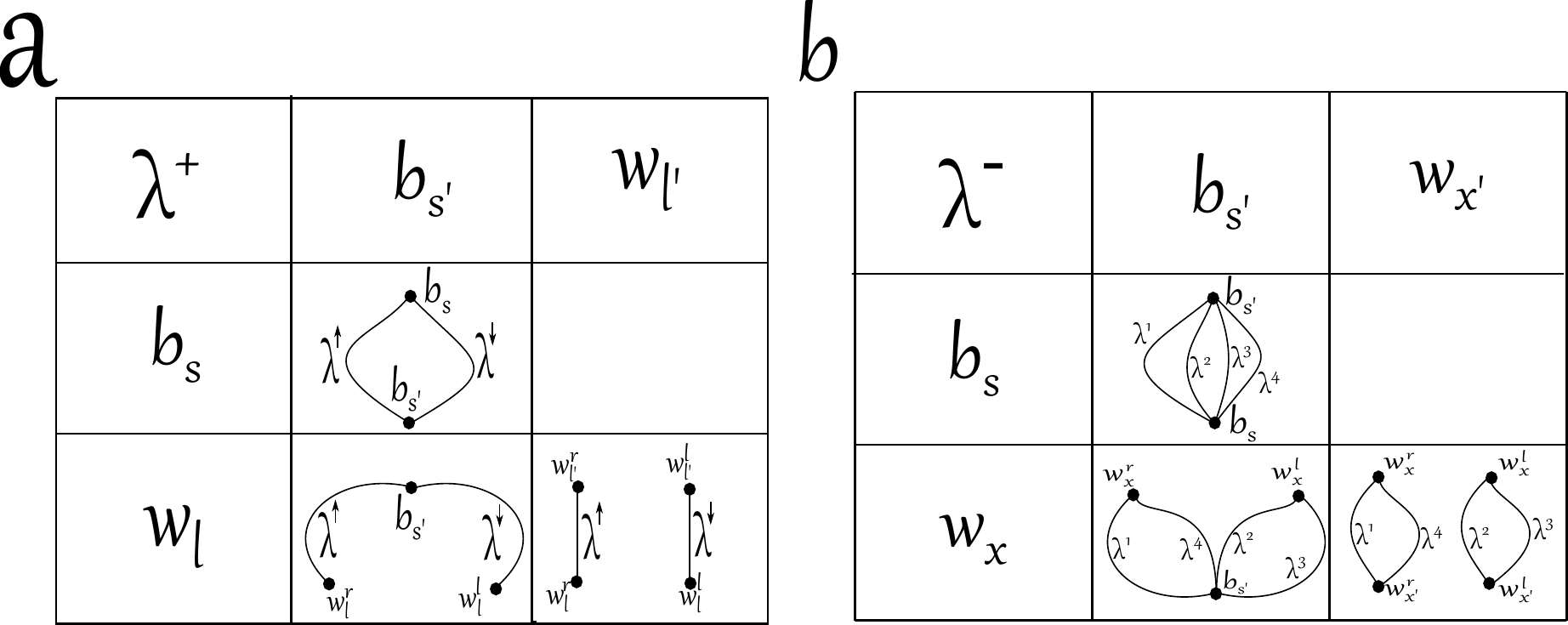}
			\caption{Таблица построения граничных графов для Алгоритмов 1, где a --- для сегментов дуги $\lambda_{i,t}$, отмеченных знаком <<+>>, а б --- для сегментов, отмеченных знаком <<->>.} 
			\label{Graphes} 
		\end{minipage}
	\end{center}
\end{figure}
\begin{comment}
Из доказательства и утверждения \ref{CutBilliard} легко видеть, что в случае, если на обоих (на правой и на левой) дугах гиперболы $\lambda_i$ лежат особые точки однородно-гиперболического биллиарда $\varOmega$, то комплекс $U_{\lambda_i} \setminus \tilde U_{\lambda_i}$ разрежется по $\nu+\xi$ ручками, где $\nu$ --- число компонент связности правой граничной дуги гиперболы $\lambda_i + \varepsilon$, а $\xi$ --- левой. Комлпекс $\tilde U_{\lambda_i}$ будет состоять из двух несвязных частей, каждая из которых построена в точности по теореме 3.
\end{comment}
\begin{comment}
Для однородно-эллиптических биллиардов, в силу односвязности, теорема 3 будет верна без изменений.
\end{comment}
\section{Особый слой $\varLambda = b$. Описание топологии атомов $U$.}

Перейдем к рассмотрению окрестности $U \subset Q^3$ особого слоя $\varLambda = b$. Совместные поверхности уровня дополнительного интеграла при $\varLambda < b$ --- несвязное объединение поверхностей рода меньшего или равного сложности биллиарда $k$, при этом траектории биллиарда на таких уровнях касаются интегральных эллипсов, прижатых к сегменту фокальной прямой между фокусами. Поверхности уровня дополнительного интеграла при $\varLambda > b$ --- несвязное объединение поверхностей рода меньшего или равного сложности биллиарда $k$, при этом траектории биллиарда касаются интегральных гипербол, прижатых к сегментам фокальной прямой вне фокусов. В классической теории (элементарных биллиардов, см. [2]) была верна теорема, описывающая строение 3-атомов $U$ через 2-атомы. Представим аналогичный результат для невыпуклых биллиардов:  сначала дадим определение 2-атомов, потом представим теорему сохранения рода, описывающую случаи, когда сумма родов поверхностей рода на совместной поверхности уровня меняется при переходе через $\varLambda = b$ и, наконец, будет представлена теорема, описывающая топологию трехмерного комплекса $U$. 

\subsection{Определение 2-атомов.}
\begin{definition}
Двумерным атомом называется пара $(P_2, K)$, где $P_2$ --- связная компактная поверхность с краем, ориентируемая или неориентируемая, а $K$ --- связный граф в ней такой, что выполняются следующие условия.
\begin{enumerate}
\item Либо K состоит только из одной точки, т.е. изолированной вершины степени ноль, либо
все вершины графа K имеют степень 4.
\item Каждая связная компонента множества $P_2\setminus K$ гомеоморфна кольцу $S_1 \times (0, 1]$ и множество этих колец можно разбить на два класса --- положительные кольца и отрицательные кольца так, так чтобы:
\item К каждому ребру графа $K$ примыкало ровно одно положительное кольцо и ровно одно отрицательное кольцо.
\end{enumerate}
\end{definition}
При этом атомы обычно рассматривают с точностью до естественной эквивалентности: два атома $(P_2, K)$ и $(P'_2, K')$ эквивалентны, если существует гомеоморфизм, переводящий $P'_2$ в $P_2$, и $K'$ в $K$.

Рассмотрим встречающиеся в плоских компактных биллиардах двумерные атомы: атом $A$, атом $B$, атом $C_2$ и атом $D_1$. В силу теоремы А.Т. Фоменко (см. теорему 3.3 в 2) трёхмерные атомы $A$, $B$, $C_2$ и $D_1$ получаются из их двумерных экземпляров прямым умножением на окружность $S^1$ (это утверждение неверно для атомов <<со звёздочками>>). 

В случае биллиардов критическая окружность на атомах $\varLambda = b$ есть в точности прообраз фокальной прямой при естественной проекции $\pi$, так как точки фокальной прямой могут быть оснащены только двумя векторами скорости на этом уровне --- от правого фокуса или к правому фокусу.   

В дальнейшем нам понадобится разрезать биллиарды по фокальной прямой, а поэтому представим <<эквивалентную>> операцию на 2-атомах. 
\begin{definition}
	Рассмотрим однородный биллиард $\varOmega$. Разрежем его по фокальной (введем на ней биллиардный закон) оси. Обозначим через $\varOmega^{top}$ --- биллиард в области выше фокальной оси, а через $\varOmega^{bot}$ --- ниже фокальной оси. Атомы этих биллиардов будем также помечать верхними индексами $top$ и $bot$.
\end{definition}
\begin{statement}
Рассмотрим элементарный биллиард $\varSigma$ без фокусов, внутренность которого имеет непустое пересечение с фокальной прямой. Пусть особый слой $G$ 3-атома $U$ биллиарда $\varSigma$ получается из 2-атома $V$ умно\-жением на окруж\-ность $S^1$. Тогда объединение особых слоев $G^{top}$ и $G^{bot}$ 3-атомов $U^{top}$ и $U^{bot}$ будет получаться домножением атома $K$, разрезанного по вершинам, на окружность.

\end{statement} 
\begin{proof}
	Рассмотрим 3-атом $U \cong V \times S^1$ и его проекцию $p: U \rightarrow V$ на 2-атом. В вершины графа $K$ в 2-атоме $V$ при проекции $p$ проектируются в точности точки, соответствующие точкам критической окружности атома $U$. Как было замечено --- это в точности точки, которые при естественной проекции $\pi$ проектируются на фокальную прямую.   
\end{proof}
Таким образом, мы можем разбить кольца, которые получаются после разрезания атома $K$ по вершинам на два класса --- проектирующиеся в биллиард $\varSigma^{top}$ и в $\varSigma^{bot}$.
\subsection{Теорема сохранения рода.}
\begin{theorem}\label{GenusSave}
	Рассмотрим биллиард $\varOmega$ и значения дополнительного интеграла $\varLambda \in [b - \varepsilon, b + \varepsilon]$. Пусть на совместных поверхностях уровня дополнительного интеграла $\varLambda < b$ лежат поверхности рода $g_1, \ldots, g_f$. А на поверхностях уровня $\varLambda > b$ лежат поверхности рода $g'_1, \ldots, g'_s$. В этом случае, $$g_1 + \ldots + g_f = g'_1 + \ldots + g'_s$$ тогда и только тогда, когда вершины углов $3\pi/2$ не лежат на фокальной прямой. 
\end{theorem}
\begin{proof}
	Согласно теореме 2 топология совместных поверхностей уровня интег\-ралов в $Q^3$ описывается количеством связных компонент области возможного движения и коли\-чест\-вом углов $3\pi/2$, лежащим внутри этих областей. А именно, каждой связной компоненте повер\-хности уровня соответствует сфера с количеством ручек (и выколотых точек), равным $k'+1$, где $k'$ --- количество вершин углов $3\pi/2$ внутри области возможного движения. Условие сохранения рода в точности означает, что количество вершин углов $3\pi/2$ в области возможного движения при фиксированном дополнительном интеграле $\varLambda = b - \varepsilon$ равно количеству вершин углов $3\pi/2$ в области возможного движения при фиксированном дополнительном интеграле $\varLambda = b + \varepsilon$. Интегральными квадриками на таких уровнях интеграла будут: эллипс с параметром близким к $b$ (<<прижатый>> к сегменту фокальной прямой между фокусами) и гипербола с параметром близким к $b$ (<<прижатая>> к участку фокальной прямой вне фокусов). Сумма родов может изменится только в том случае, если вершины углов $3\pi/2$ лежат внутри такого интегрального или такой интегральной гиперболы. В силу произвольности $\varepsilon$ это условие означает, что вершины углов обязаны лежать на фокальной прямой.
\end{proof}
\begin{comment}
В случае, когда вершины углов $3\pi/2$ лежат на фокальной прямой перестройки двух разных типов (с падением ранга и без) происходят <<одновременно>> на уровне $\varLambda = b$.
\end{comment}
\subsection{Основная теорема.}
Аналогично случаю $\varLambda = \lambda_i$ определим двумерный подкомплекс $T_i$ в трехмерном комплексе $U$, который будет соответствовать дуге граничной гиперболы $\lambda_i$. 
\begin{definition}
		Рассмотрим однородный биллиард $\varOmega$ с выбранным на нем разбиением $\varSigma_1, \ldots, \varSigma_N$ и рассмотрим двумерную клетку $T_i \subset U$, состоящую из таких пар $(x,v)$, которые проектируются естественной проекцией $\pi$ на дугу квадрики $\lambda_i$. Клетки выбраны как показано в главе 5. 
\end{definition}
Рассмотрим элементарные биллиарды разбиения $\varSigma_1, \ldots, \varSigma_N$. Топология 3-атомов $U$ для каждого из таких элементарных биллиардов полностью описывается теоремой \ref{V}. Воспользуемся теоремой и обозначим соответствующие 2-атомы для таких биллиардов через $V_{j}$. 

Рассмотрим однородный биллиард $\varOmega$, каустику с параметром $\lambda = b - \varepsilon$ и интегральную гиперболу семейства с параметром $\lambda = b + \varepsilon$. Теперь рассмотрим области возможного движения при таких значениях дополнительного интеграла $\varLambda$. Обозначим область возможного движения для $\varLambda = b - \varepsilon$ (область биллиарда $\varOmega$ вне эллипса, <<близкого>> к фокальной оси) через $\varOmega_{b-\varepsilon}$. Аналогично область возможного движения при значениях интеграла $\varLambda = b + \varepsilon$ (между дуг гиперболы, <<прижатых>> к фокальной прямой) обозначим через $\varOmega_{b+\varepsilon}$.
Заметим, что области $\varOmega_{b+\varepsilon}$ и $\varOmega_{b-\varepsilon}$ могут оказаться несвязанными. 

Теперь рассмотрим разбиение $\varSigma_1, \ldots, \varSigma_{N}$ биллиарда $\varOmega$. Выберем разбиения областей возможного движения $\varOmega_{b+\varepsilon}$ и $\varOmega_{b-\varepsilon}$ на элементарные билиларды $\varSigma^{>}_1, \ldots, \varSigma^{>}_{N'}$ и $\varSigma^{<}_1, \ldots, \varSigma^{<}_{N''}$ соответственно (см. опр. разбиения). 

Сформулируем основной результат --- теорему описания топологии седлового особого слоя для однородного биллиарда. Графическую интерпретацию теоремы можно найти на рис. \ref{Table2}.
\begin{theorem}
	Рассмотрим однородно-гиперболический биллиард $\varOmega$ с выбранным на нем разбиением $\varSigma_1, \ldots, \varSigma_N$. Рассмотрим окрестность значений дополнительного интеграла $b - \varepsilon \leq \varLambda \leq b + \varepsilon$ и соответствующий этой окрестности 3-атом $U$. Вырежем из трехмерного комплекса $U$ все двумерные подкомплексы $T_i$ для всех дуг граничных гипербол $\lambda_1, \ldots, \lambda_n$ разбиения $\varSigma_1, \ldots, \varSigma_N$. Тогда:
	\begin{enumerate}
		\item Двумерный комплекс $T_i$ расслаивается на одномерные графы следующим образом: $T_i |_{\varLambda < b} \cong Gr^{<}_{i}$, $T_i |_{\varLambda = b} \cong Gr^{=}_{i}$, $T_i |_{\varLambda > b} \cong Gr^{>}_{i}$, где графы $Gr^{<}_{i}$, $Gr^{=}_{i}$, $Gr^{>}_{i}$ строятся по алгоритму 4;
		\item Трехмерный комплекс $U \setminus (T_{1} \cup \ldots \cup T_{n}) \cong 2(V_{1} \times I) \cup \ldots \cup 2(V_{N} \times I)$, где объединение несвязно в силу однородности биллиарда $\varOmega$;
		\item Двумерные комплексы $(T_{1} \cup \ldots \cup T_{n})$ приклеиваются к трехмерному комплексу $U \setminus (T_{1} \cup \ldots \cup T_{n})$ послойно (где на 2-атомах выбрано слоение на окружности), таким образом, к графам $Gr^{<}_{i}$, $Gr^{=}_{i}$, $Gr^{>}_{i}$ приклеиваются произведение колец атомов $V_{\varSigma_j}$ на окружности (цилиндры) и произведение графов $K$ этих атомов на окружность. Данная склейка описывается алгоритмом 5.
	\end{enumerate}
\end{theorem}

\begin{figure}[h!]
	\begin{center}
		\begin{minipage}[h]{1.0\linewidth}
			\includegraphics[width=0.9\linewidth]{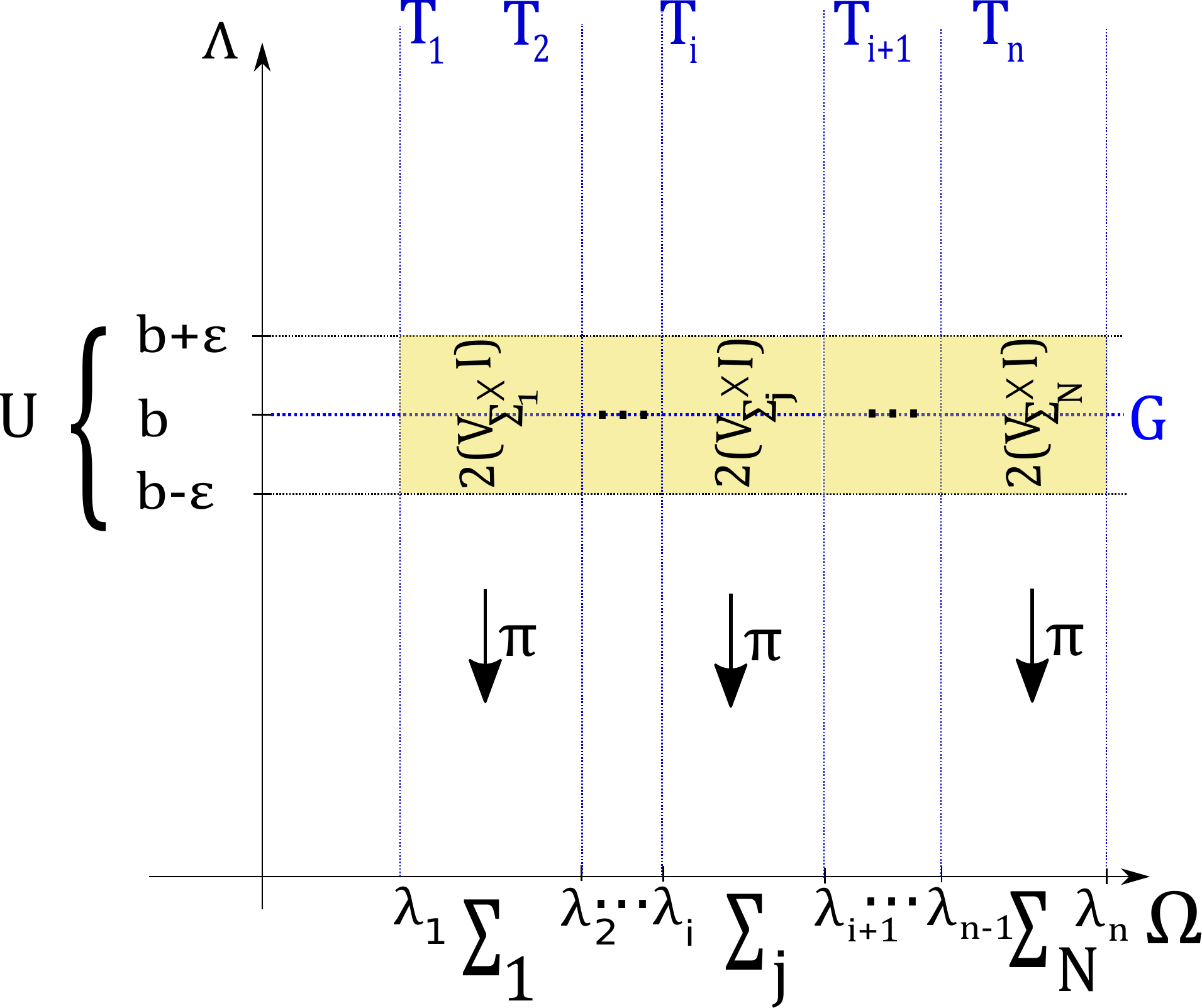}
			\caption{Иллюстрация теоремы 5. По оси абсцисс обозначены значения естественной проекции $\pi$, а по оси ординат --- значения дополнительного интеграла $\varLambda$.} 
			\label{Table2} 
		\end{minipage}
	\end{center}
\end{figure}
\begin{figure}[h!]
	\begin{center}
		\begin{minipage}[h]{1.0\linewidth}
			\includegraphics[width=0.9\linewidth]{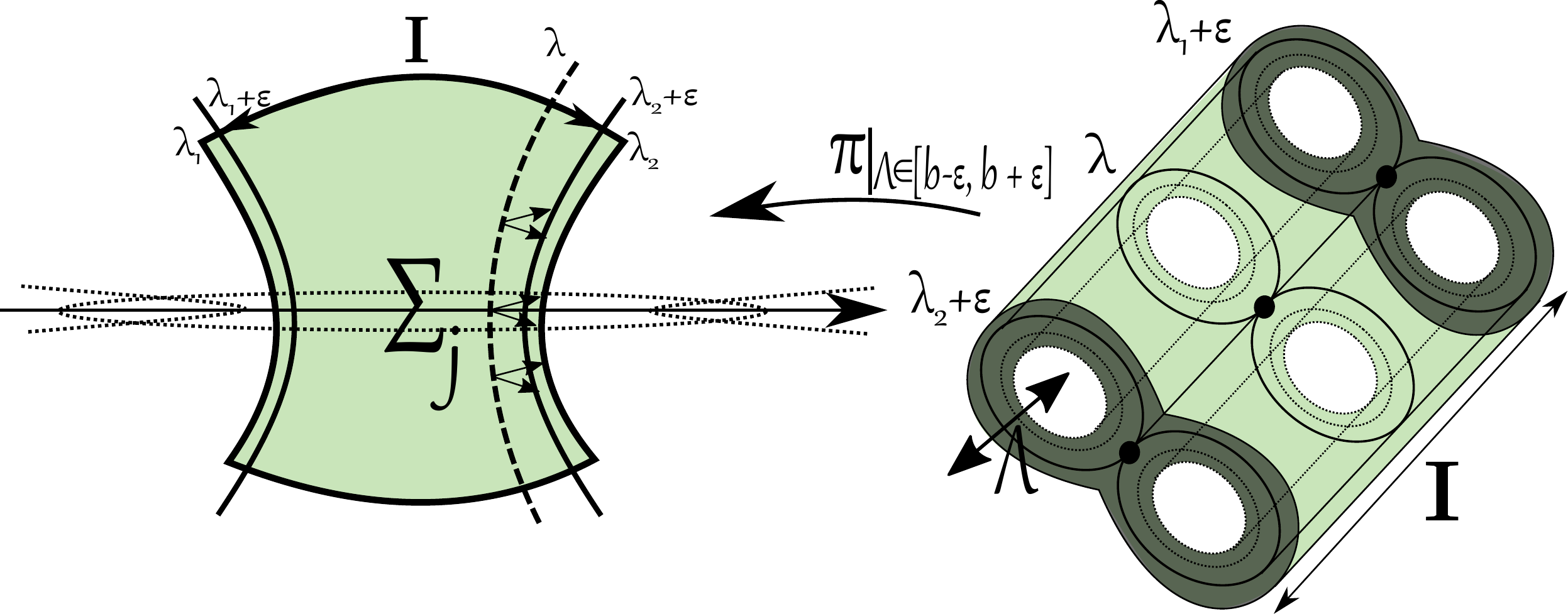}
			\caption{Пример проекции двумерной клетки из $U$ на элементарный биллиард $\varOmega$ разбиения (без сегментов граничных гипербол). Двумерная клетка гомеоморфна $B \times I$, где $B$ --- двумерный атом. В образ клетки попадают только точки, оснащенные векторами, направленными вправо.  } 
			\label{FiberB} 
		\end{minipage}
	\end{center}
\end{figure}

\textbf{Алгоритм 4.}
\begin{figure}[h!]
	\begin{center}
		\begin{minipage}[h]{1.0\linewidth}
			\includegraphics[width=0.9\linewidth]{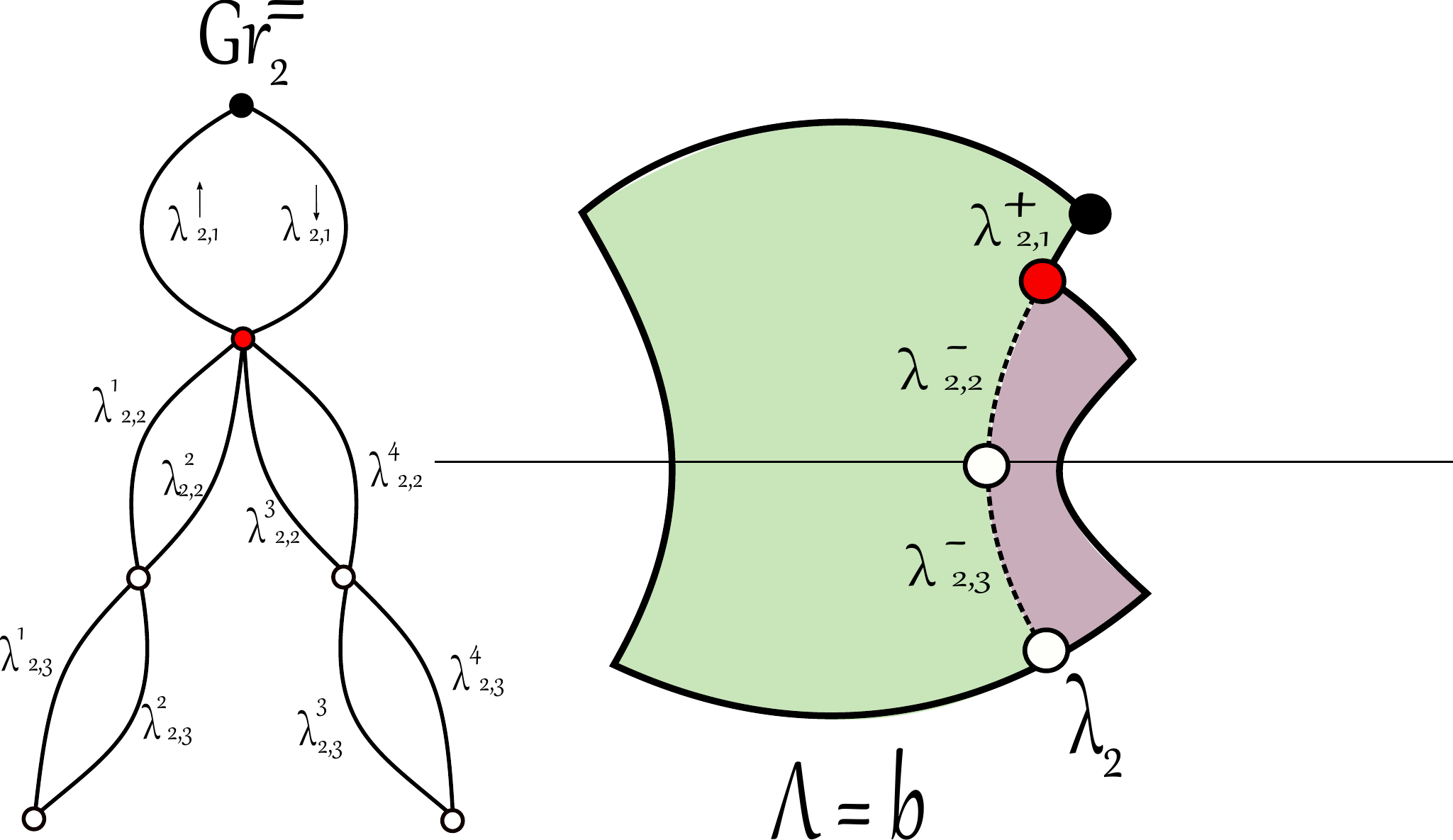}
			\caption{Пример построения графа $Gr^{=}_2$ по Алгоритму 4 для дуги граничной гиперболы $\lambda_2$ биллиарда, изображенного на рисунке.} 
			\label{Alg4} 
		\end{minipage}
	\end{center}
\end{figure}
\begin{enumerate}
	\item Рассмотрим биллиард $\varOmega$ и его разбиение на элементарные биллиарды $\varSigma_1, \ldots, \varSigma_N$. Рассмотрим дуги граничных гипербол такого разбиения $\lambda_1, \ldots, \lambda_n$. Теперь рассмотрим области $\varOmega_{b+\varepsilon}$ и $\varOmega_{b-\varepsilon}$ и их разбиения на элементарные биллиарды $\varSigma^{>}_1, \ldots, \varSigma^{>}_{N}$ и $\varSigma^{<}_1, \ldots, \varSigma^{<}_{N}$;
	\item Построим графы $Gr^{>}_{i}$ и $Gr^{<}_{i}$, так как слои $b-\varepsilon$ и $b + \varepsilon$ не критические, то воспользуемся Алгоритмом 1 для их построения, применив его дважды: для области возможного движения $\varOmega_{b-\varepsilon}$ с зафиксированным дополнительным интегралом $\varLambda = b - \varepsilon$ и для области возможного движения $\varOmega_{b+\varepsilon}$ с зафиксированным дополнительным интегралом $\varLambda = b + \varepsilon$;
	\item Теперь построим граф $Gr^{=}_{i}$. Воспользуемся модифицированным Алгоритмом 1 для дуги граничной гиперболы $\lambda_i$ и биллиардов элементарного разбиения $\varSigma_1, \ldots, \varSigma_{N}$: на шаге 1 дополнительно отметим белыми точками пересечения граничной дуги гиперболы $\lambda_i$ с фокальной прямой. 
\end{enumerate}
\textbf{Алгоритм 5.}
\begin{enumerate}
	\item Рассмотрим элементарный биллиард $\varSigma^{>}_j$ из разбиения $\varSigma^{>}_1, \ldots, \varSigma^{>}_{N'}$. Также рассмотрим положительные кольца $S^1$ 2-атома $V_j$ и построенные алгоритмом 3 графы $Gr^{>}_{i}$ и $Gr^{>}_{i'}$. Здесь $\lambda_i$ и $\lambda_j$ суть границы элементарного биллиарда $\varSigma^{>}_j$. Поставим в соответствие каждому положительному кольцу $S^1$ по два цилиндра $S^1 \times I$ и обозначим их через $\varSigma_j^{R, >}$ и $\varSigma_j^{L, >}$. Применим алгоритм 2 к элементарному биллиарду $\varSigma^{>}_j$ и графам $Gr^{>}_{i}$ и $Gr^{>}_{i'}$.
	\item Аналогично применим алгоритм 2 для элементарного биллиарда $\varSigma^{<}_j$ и графам $Gr^{<}_{i}$ и $Gr^{<}_{{i'}}$.
	\item Разделим биллиард $\varOmega$, 2-атомы всех элементарных биллиардов $\varSigma_j$ и все графы $Gr^{=}_{i}$ на две части. Проведем фокальную прямую и разобьем биллиард $\varOmega$ на два биллиарда: $\varOmega^{top}$ и $\varOmega^{bot}$. Теперь рассмотрим все элементарные биллиарды $\varSigma_j$ из разбиения $\varSigma_1, \ldots, \varSigma_N$ и соответствующие им двумерные атомы $V_j$ c графами $K_j$. Теперь разрежем граф $K_j$ по вершинам и поставим биллиардам $\varOmega^{top}$ и $\varOmega^{bot}$ в соответствие получившиеся кольца $K^{top}_j$ и $K^{bot}_j$ (см. главу 6.1). Граф $Gr^{=}_{i}$ мы также разрежем на два графа $GrTop^{=}_{i}$ и $GrBot^{=}_{i}$ по точкам (белым или черным), соответ\-ствующим точкам пересечения граничной дуги гиперболы $\lambda_i$ с фокальной прямой. 
	
	\item Рассмотрим получившиеся в результате разреза графа 2-атома $K_j$ кольца. Домножим эти кольца на отрезок и рассмотрим два экземпляра итогового произведения. Если на кольце лежала вершина графа $K$, то поставим на одном из ребер получившегося цилиндра букву $\lambda_{0,t}$, где индексом $t$ соответственно пронумеруем участки фокальной прямой между дугами $\lambda_i$. Обозначим получившиеся цилиндры через $$C^R_{1,j}, \ldots, C^R_{\nu,j},C^L_{\nu +1,j}, \ldots, C_{2\nu,j}$$ Разделим получившиеся цилиндры на два класса --- относящиеся к биллиарду $\varOmega^{top}$ (такие пометим символом $~$)  и на относящиеся к биллиарду $\varOmega^{bot}$ (см. главу 6.1).
	
	\item Теперь будем применять Алгоритм 2 по отдельности к биллиардам $\varOmega^{top}$ и $\varOmega^{bot}$. А именно, при помощи Алгоритма 2 приклеим все графы $GrTop^{=}_{i}$ к тем цилиндрам, которые на предыдущем шаге были помечены волной. Аналогично приклеим остальные цилиндры к графам $GrBot^{=}_{i}$. Обозначим получившиеся в результате склейки комплексы $U^{top}$ и $U^{bot}$ соответственно: это суть 3-атомы биллиардов $\varOmega^{top}$ и $\varOmega^{bot}$.  
	\item Теперь склеим $U^{top}$ и $U^{bot}$ по одинаковым буквам $\lambda_{0,t}$.

\end{enumerate}
\begin{proof}
	Рассмотрим однородно-гиперболический биллиард $\varOmega$ и дополнительный интеграл $\varLambda \in [b-\varepsilon, b+ \varepsilon]$.
	
	Сначала убедимся, что Алгоритм 4 действительно строит двумерные подкомплексы $T_i$. Пусть точка $x \in \varOmega$ гиперболы $\lambda_i$ попадает в область возможного движения при некотором значении интеграла  $\varLambda$ близкого или равного $b$, тогда: 
	\begin{enumerate}
	
	\item Если точка $x$ не лежит на границе биллиарда $\varOmega$ или на фокальной прямой, то в ней определены 4 вектора скорости, направленные либо по касательной к каустике ($\varLambda \neq b$), либо по траекториям, проходящим через фокусы ($\varLambda = b$). В алгоритме такая точка $x$ лежит на отмеченном знаком <<$-$>> сегменте;
	\item Если точка $x$ лежит на границе биллиарда $\varOmega$, то, в силу биллиардного закона, в такой точке определены только два вектора скорости. В алгоритме такая точка $x$ лежит на отмеченном знаком <<$+$>> сегменте;; 
	\item Если точка $x$ лежит на фокальной прямой, то на уровне $\varLambda = b$ в ней определены только два вектора скорости, а на всех остальных уровнях дополнительного интеграла --- 4. В алгоритме такие точки дополнительно отмечены белым; 
	\item Если точка $x$ лежит в вершине прямого угла, то ей соответствует один вектор скорости, а если в вершине угла $3\pi/2$, то все пары $(x,v)$ склеятся в $Q^3$. В алгоритме такие точки дополнительно отмечены черным; 
	\item Если точка $x$ границы биллиарда $\varOmega$ лежит на каустике $\lambda = \alpha$, то в ней также будет определён лишь один вектор скорости,  в силу закона отражения. Такие точки в алгоритме также отмечены черным;
	\item Если точка внутренняя точка $x$ биллиарда $\varOmega$ лежит на каустике $\lambda = \alpha$, то в ней также будет определено два вектора скорости,  в силу закона отражения. Такие точки в алгоритме также отмечены белым;
	\end{enumerate}
	Рассмотрим прообраз дуги $\lambda_i$ при фиксированном значении дополнительном интеграла $\varLambda = \alpha \in [b-\varepsilon, b+ \varepsilon]$ и при фиксированном значении интеграла $\varLambda = \beta \in [b-\varepsilon, b+ \varepsilon]$. Обозначим эти (одномерные) прообразы через $Gr_\beta$ и $Gr_\alpha$ соответственно. Заметим, что структура $Gr_\beta$ и $Gr_\alpha$ зависит только от области возможного движения при $\varLambda = \alpha$ и $\varLambda = \beta$, а она, в силу малости $\varepsilon$, зависит только от соотношений $\alpha, \beta$ и $b$. А именно,  $Gr_\beta \cong Gr_\alpha$ в том и только том случае, если либо $\alpha < b$ и $\beta < b$, $\alpha = \beta = b$, либо $\alpha > b$ и $\beta > b$.
	В итоге, алгоритм 4 действительно строит графы $Gr^{<}_{i}$, $Gr^{=}_{i}$ и $Gr^{>}_{i}$, так как структура одномерных комплексов границ меняется только при переходе через критическое значение. 
	
	Теперь перейдём к доказательству второго утверждения теоремы 5:  рассмотрим комплекс $A = U \setminus (T_{1} \cup \ldots \cup T_{n})$. Зафиксируем некоторое значение интеграла $\varLambda = \alpha \in  [b-\varepsilon, b+ \varepsilon]$ и спроецируем комплекс $A$ на биллиард $\varOmega$ естественной проекцией $\pi$. В итоге получится несвязное объединение элементарных биллиардов $\varSigma_1, \ldots, \varSigma_{N}$ без границ. Рассмотрим один из таких биллиардов $\varSigma_j$ и введём биллиардный закон на его границах: это будет плоский элементарный биллиард, а потому в его прообразе при естественной проекции будет слой $\varLambda = \alpha$ одного из атомов. Тогда, варьируя $\alpha$, мы получаем, что прообраз $\varSigma_j$ при $\varLambda \in [b-\varepsilon, b+ \varepsilon]$ и введённом биллиардном отражении на <<вырезанных>> границах, гомеоморфен 3-атому из списка теоремы 3, а точнее, гомеоморфен $V_j \times S^1$. Теперь уберем биллиардное отражение на границе элементарного биллиарда $\varSigma_j$: каждый слой $V_j$ гомеоморфен прообразу оснащённой векторами скорости гиперболы, а значит, при убирании двух таких многообразие $V_j \times S^1$ разобьётся на два многообразия $V_j \times I$, что и требовалось доказать.
	
	Перейдем к заключительному пункту теоремы 3. Докажем, что алгоритм 5 действительно корректно склеивает двумерные клетки $T_i$ с трехмерными клетками $2(V_j \times I)$. Для всех слоев интеграла $\varLambda \neq b$ доказательство полностью аналогично доказательству теоремы 3, так как эти слои не являются особыми. Для построения особого слоя $G$ атома $U$ сначала разрежем биллиард $\varOmega$ по фокальной прямой и, следуя алгоритму, рассмотрим биллиарды $\varOmega^{top}$ и $\varOmega^{bot}$. Данные биллиарды не пересекаются с фокальной прямой внутри области, а потому его можно разбить на элементарные биллиарды вида $\varSigma_1 \cap \varOmega^{top}, \ldots, \varSigma_N \cap \varOmega^{top}$ и $\varSigma_1 \cap \varOmega^{bot}, \ldots, \varSigma_N \cap \varOmega^{bot}$ и любой элементарный биллиард $\varSigma_j \cap \varOmega^{top}$ или $\varSigma_j \cap \varOmega^{bot}$ будет биллиардом без особенности на уровне $\varLambda = b$ (иначе говоря, 3-атом гомеоморфен произведению тора на отрезок). Способ приклейки таких биллиардов к графам $Gr^{>}_{i}$, $Gr^{=}_{i}$ или $Gr^{<}_{i}$ опять же аналогичен теореме 3. Дополнительно в Алгоритме 4 помечаются ребра, по которым производился разрез 2-атомов: именно склейка по ним отвечает склейки по фокальной прямой на уровне атомов.

\end{proof}
\begin{comment}
Из доказательства теоремы 5 следует, что если биллиард $\varOmega$ не содержит сегментов форкальной прямой внутри или на границе области, то в таком биллиарде $U \cong G^k_2 \times I$, где $G^k_2$ --- 2-поверхность рода $k$, где $k$ --- сложность биллиарда. Так как в таком случае Алгоритм 3 и Алгоритм 5 в точности повторяют Алгоритмы 1 и 2 для не седловых слоев.  
\end{comment}
\begin{comment}
Из теоремы 5 следует, что особый слой невыпуклого биллиарда $U$ склеивается из особых слоев выпуклых биллиардов $\varSigma_j$, входящих в его разбиение. Однако, если выбрать разбиение биллиарда $\varOmega$ каким-либо другим способом (например, разрезать по всем квадрикам на которых лежат особые точки), то это разбиение поднимется до другого разбиения комплекса $U$ на клетки. Например, биллиард, изображенный на рисунке 20 разбивался на 2 элементарных биллиарда типа $A_0$, однако, его можно разбить на 3 элементарных биллиарда: 2 типа $A_0$ и на один элементарный биллиард типа $B_0$, проведя элллипс, на котором лежит особая точка. Это приводит к невозможности описания топологии атома $U$ только через описание двумерных клеток этого комплекса. 
\end{comment}

\section{Заключение.}
Предметом анализа данной работы были невыпуклые плоские биллиарды, а именно, исследовалась топология слоения Лиувилля для многообразия $Q^3$ в таких биллиардах. В ходе работы невыпуклый биллиард был сначала разбит на элементарные биллиарды (без углов $3\pi/2$), а потом полученное разбиение было <<поднято>> на уровень $Q^3$.  В теоремах 4 и 5 было представлено полное алгоритмическое описание топологии окрестности особого слоя критических значений дополнительного интеграла. Теорема 3 описывает окрестности особых слоев $\varLambda = \lambda_i$, а Теорема 5 --- топологию окрестности особого слоя $\varLambda = b$. Отметим, что ранее описание окрестностей критических значений не встречалось в других работах по данной тематике. 

Результаты, полученные в данной работе, докладывались на международных конфе\-ренциях (Ломоносов 2019, Ломоносов 2020, Воронежская Зимняя математическая школа С.Г. Крейна, Equadiff-2019, GDIS 2018, CIS 2018) и научно-исследовательском семинаре <<Совре\-менные геометрические методы>> под руководством акад. А. Т. Фоменко, проф. А. С. Мищенко, проф. А. В. Болсинова, проф. А. А. Ошемкова, проф. Е. А. Кудрявцевой, доц. И. М. Никонова, доц. А.Ю. Коняева, асс. В. В. Ведюшкиной (механико-матема\-тический факультет МГУ имени М. В. Ломоносова).
  
\section{Список литературы.}
\begin{enumerate}
\item {Табачников С.\,Л.}
Геометрия и бильярды. М.; Ижевск: НИЦ \textquotedblleft РХД\textquotedblright, 2011. 
\item {Болсинов А.В., Фоменко А.Т.} \label{Fom}
Интегрируемые гамильтоновы системы. Геометрия, топология, классификация. Т. 1. Ижевск: НИЦ \textquotedblleft РХД\textquotedblright, 1999. 
\item {Козлов В. В.} Некоторые интегрируемые обобщения задачи Якоби о геодезических на эллипсоиде // Прикладная математика и механика, том 59, вып. 1, 1995.
\item {Фокичева В. В.} \label{V} Топологическая классификация бильярдов в локально плоских областях, ограниченных дугами софокусных квадрик//
Матем. сб. 2015. \textbf{206}, \textnumero 10. 127--176.
\item {Фокичева В.В., Фоменко А.Т.} Интегрируемые бильярды моделируют важные интегрируемые случаи дина\-мики твердого тела // Докл. РАН. Сер. матем. 2015. \textbf{465}, \textnumero 2. 150---153. (Integrable Billiards Model Important Integrable Cases of Rigid Body Dynamics // Doklady Mathematics. 2015. \textbf{92}, N 3. 1---3. Pleiades Publishing, Ltd., 2015). 
\item {Ведюшкина В.В., Фоменко А.Т.} Интегрируемые топологические бильярды и эквивалентные динами\-ческие системы //Изв. РАН. Cер. матем. 2017. \textbf{81}, \textnumero 4. 20---67. 
\item {Fokicheva V., Fomenko T.} Billiard Systems as the Models for the Rigid Body Dynamics // Studies in Systems, Decision and Control. Advances in Dynamical Systems and Control. Vol.69. Ed. by V. Sadovnichiy, M. Zgurovsky. Springer; International Publishing Switzerland, 2016. 13---32. 
\item {Dragovic V., Radnovic M.} \label{DR}
Bifurcations of Liouville tori in elliptical billiards// 
Regular Chaotic Dyn. РАН. 2009.\textbf{ 14}. 479---494. 
\item {Драгович B., Раднович М.} Интегрируемые бильярды, квадрики и многомерные поризмы Понселе. М.; Ижевск: НИЦ \textquotedblleft РХД\textquotedblright, 2010. 
\item {Dragovic V., Radnovic M.} Pseudo-integrable billiards and arithmetic dynamics // Modern Dynamics. 2014. \textbf{8}, N1, 109---132. 
\item {Dragovic V., Radnovic M.} Pseudo-integrable billiards and double-reflection nets // Russ. Math. Surveys. 2015. \textbf{70}, N{1}, 1---31. 
\item {Dragovic V., Radnovic M.} Periods of pseudo-integrable billiards // Arnold Math. 2015. \textbf{1}, N{1}. 69---73. 
\item {Bolsinov A.V., Fomenko A.T., Oshemkov A.A.} Topological Methods in the Theory of Integrable Hamiltonian. Cambridge: Cambridge Scientific Publishers, 2006. 
\item {Кудрявцева Е.А., Никонов И.М., Фоменко А.Т.} Максимально симметричные клеточные разбиения поверхностей и их накрытия //Матем. сб. 2008. \textbf{199}, \textnumero 9. 3---96. 
\item {Кудрявцева Е.А., Никонов И.М., Фоменко А.Т.} Симметричные и неприводимые абстрактные многогранники // Современные проблемы математики и механики. Т. 3. Математика. Вып. 2. Геометрия и топология // Изд-во ЦПИ при мех.-мат. ф-те МГУ, 2009. 58---97. 
\item {Кудрявцева Е.А., Фоменко А.Т.} Группы симметрий правильных функций Морса на поверхностях // Докл. РАН. Cер. матем. 2012. \textbf{446}, \textnumero 6. 615---617. 
\item {Кудрявцева Е.А., Фоменко А.Т.} Любая конечная группа является группой симметрий некоторой карты (\textquotedblleft атома\textquotedblright - бифуркации) // Вестн. Моск.ун-та. Матем. Механ. 2013. \textnumero 3. 21---29. 
\item {Кудрявцева Е.А.} Аналог теоремы Лиувилля для интегрируемых гамильтоновых систем с неполными потоками// Докл. РАН. 2012. \textbf{445}, \textnumero 4. 383---385. 
\item {Москвин В.А.} Топология слоений Лиувилля интегрируемого биллиарда в невыпуклых областях// 
Вестник МГУ, 2018, \textbf{3}. 21 --- 29.
\item {Москвин В.А.} Алгоритмическое построение двумерных особых слоев атомов бильярдов в невыпуклых областях//
Вестник МГУ, 2020, \textbf{3}, 3 --- 12.
\end{enumerate}
\end{document}